\documentclass{amsart}
\usepackage{graphicx} 
\usepackage{amsmath, amsthm, amssymb}
\usepackage{tikz-cd}
\usepackage{tikz}
\usepackage{stmaryrd}
\usepackage{relsize}
\usepackage{enumitem} 
\usepackage{ragged2e}
\usepackage{microtype}
\usepackage{amsfonts}
\usepackage{extarrows}
\usepackage{extarrows}
\DeclareMathOperator{\Hom}{Hom}
\DeclareMathOperator{\Tor}{Tor}

\DeclareMathOperator{\rank}{rank}
\usepackage{graphicx}
\usepackage{amsthm}
\usepackage{csquotes}
\usepackage{tikz-cd}
\usepackage{longtable}
\usepackage{array}
\usepackage{capt-of}
\usepackage{MnSymbol}
\usepackage{lscape}
\usepackage{float}
\usepackage{hyperref}
\usepackage{url}
\usepackage{rotating}
\usepackage{color}
\usepackage{graphicx}
\usepackage[style=numeric,doi=false,url=false]{biblatex}
\usepackage{etoolbox}


\def\HH{\operatorname{HH}}
\title[Koszul Artin Schelter-Regular Algebras of Dimension four]{Some components of the moduli space of Koszul Artin-Schelter regular algebras of dimension four}
\author{Vishal Bhatoy}
\address{Department of Mathematics, Carleton University, Ottawa, Ontario, Canada}
\email{vishalbhatoy@cmail.carleton.ca}

\author{Colin Ingalls}
\address{Department of Mathematics, Carleton University, Ottawa, Ontario, Canada}
\email{coliningalls@cunet.carleton.ca}

\author{F\'elix LaRoche}
\address{Department of Mathematics, Carleton University, Ottawa, Ontario, Canada}
\email{FelixLaroche@cmail.carleton.ca}

\author{Ravali Nookala}
\address{Department of Mathematics, Carleton University, Ottawa, Ontario, Canada}
\email{ravalinookala@cmail.carleton.ca}

\begin{document}
\begin{abstract}

    We compute the Hochschild cohomology and the Kodaira spencer map for known families of Koszul Artin-Schelter regular algebras of dimension four.  We show that when the Kodaira Spencer map at a point is a surjection, the image of the family is a component of the moduli stack of such algebras, and when the Kodaira Spencer map is a bijection, the map to the moduli stack is generically finite. We use this to identify some components of the moduli stack.
\end{abstract}
\maketitle
\newtheorem{theorem}{Theorem}[section]  
\newtheorem{lemma}[theorem]{Lemma}       
\newtheorem{proposition}[theorem]{Proposition}  
\newtheorem{corollary}[theorem]{Corollary}      
\theoremstyle{definition}
\newtheorem{definition}[theorem]{Definition}    
\newtheorem{remark}[theorem]{Remark}            
\newtheorem{example}[theorem]{Example}           
\newtheorem{question}[theorem]{Question}         
\newtheorem{custom_def}[theorem]{} 
\newtheorem*{mainthm}{Theorem}
\newcommand{\BigOplus}{\mathop{\scalebox{1.3}{$\bigoplus$}}}

\section{Introduction}

For some known families of Artin-Schelter regular algebras, it is possible to show that infinitesimal deformations of generic members of the families stay in the given family.  We show this as an example in ~\ref{toricfamily}.  This example can be formalised as a computation of the Kodaira-Spencer map from the tangent space of the base of the family at a point to the Hochschild Cohomology of the algebra.  We compute this for known families of Koszul Artin-Schelter algebras of dimension four.  This allows us to conclude which known families of Koszul Artin-Schelter regular algebras of dimension four form components of the moduli space.  
In~\cite{ARTIN1987171}, Artin and Schelter defined a class of noncommutative analogues of commutative polynomial algebras. They classified all such algebras of dimension three and less.  These Artin-Schelter (AS) regular algebras have been studied extensively \cite{rogalski2023artinschelter} and their classification has been approached from different perspectives \cite{Okawa,Artin2007SomeAA, MR1128218,BondalPolischuk}.  AS-regular algebras of higher dimensions have not been classified.  
There are many known families and some partial classifications.  Notably, AS-regular algebras of dimension three that are double Ore extensions are classified in \cite{MR2529094}. This paper constructs 26 families of AS-regular algebras. AS-regular algebras of dimension four fall into three types based on the ranks of the modules in the projective resolution of the residue field, $\textit{14641}$, $\textit{12221}$, and $\textit{13431}.$ The class $\textit{12221}$ has largely been classified by \cite{Lu_2007}.  The other known examples in the literature fall into the class $\textit{14641}$ of Koszul Artin-Schelter (KAS) regular algebras.

Given a flat family of Algebras $A_S$ with parameter space $S$, and a point $p\in S$, we use GAP code to compute the degree zero component of the second Hochschild Cohomology $\HH^2(A_p)_0$.  We further compute the infinitesimal deformations of $A_p$ given by infinitesimal perturbations of the point $p \in S$.  This gives the Kodaira-Spencer map
$$T_pS \to \HH^2_0(A_p).$$

We describe the moduli space of Koszul AS-regular algebras of dimension four via the Koszul dual.  In particular, we define $\mathcal{A}_4$ to be the moduli stack of Frobenius graded algebras with Hilbert Series $(1+t)^4$.  When these algebras are Koszul, they are Koszul duals of KAS regular algebras of dimension four.  However, the Koszul condition imposes infinitely many algebraic conditions, so we do not use this restriction.

When the Kodaira-Spencer map is a surjection, we can conclude that the closure of the image of the induced map $f : S \to \mathcal{A}_4$ is a component of the moduli stack.  

{In Section~2,} we review the necessary background and definitions, including Artin-Schelter regular algebras, their minimal resolutions, and the classification of four-dimensional resolution types. We also discuss Hochschild cohomology and its bigraded refinement, Koszul algebras and duality, and the framework of moduli spaces and stacks for studying families of algebras.

{In Section~3,} we study first-order deformations of flat families of graded algebras over a base scheme $S$ via Hochschild cohomology. We construct the \emph{Kodaira--Spencer map}, which assigns to each tangent vector $v \in T_p S$ a class in the degree-zero part of $\HH^2_0(A_p)$, where $A_p$ is the fibre at $p$. This map encodes how infinitesimal base variations induce first-order algebra deformations of the fibre.

We study an accessible example,
the skew polynomial algebras $A = \mathbb{K}_Q[x_1, x_2, x_3, x_4]$ determined by a generic multiplicatively skew-symmetric matrix $Q$. We prove that the Kodaira--Spencer map is an isomorphism. Consequently, each tangent vector corresponds uniquely to a graded first-order deformation, and the family forms a smooth component of the moduli space of four-dimensional AS-regular algebras. This establishes a concrete link between Hochschild cohomology and the local geometry of the moduli space.

 {In Section~4,} we present the computations of  the Hochschild cohomology $\HH_0^i$ and the Kodaira--Spencer map for known families of four-dimensional Koszul Artin--Schelter regular algebras of type $\textit{(14641)}$. For each family, we evaluate these maps at specific parameter values, recording dimensions, injectivity, and surjectivity.

{In Section~5,} we establish criteria ensuring that a morphism from an irreducible scheme or family of algebras to an algebraic stack $\mathcal{Y}$ maps densely onto an irreducible component. We begin by showing that if the differential at a point is surjective, the image closure of a morphism between schemes is an irreducible component in Propositions~\ref{prop:irreducible_component} and~\ref{prop:image_irreducible}. We extend these results to algebraic stacks by proving that surjectivity of the tangent map implies that the family maps onto an irreducible component in Theorem~\ref{thm:irreducible_component}. Applying this to the universal family of four-dimensional Koszul Artin--Schelter regular algebras, we identify the Zariski tangent space at a point with $\HH^2_0(A_p)$ and deduce that surjectivity of the Kodaira--Spencer map guarantees that the image of each computed family is dense in an irreducible component of the moduli stack $\mathcal{A}_4$ in Corollary~\ref{corr510} and Theorem~\ref{mainthm511}.

These computations allow us to identify which families map densely onto irreducible components of the moduli stack $\mathcal{A}_4$. 

We thank Michaela Vancliff and Ellen Kirkman for providing some algebra families used in our computations.

\section{Background and Definitions}
Let $\mathbb{K}$ be an algebraically closed field of characteristic zero.
We will say $A$ is connected graded if $A$ is graded $A = \oplus_{i \geq 0} A_i$ with $A_0=\mathbb{K}.$

\begin{definition}
\textbf{Artin-Schelter regular algebra} \cite{ARTIN1987171}: An Artin-Schelter regular algebra (AS-regular algebra) is a connected graded algebra $A$ over a field $\mathbb{K}$ generated in degree one, such that the following conditions hold:
\begin{enumerate}
    \item $A$ has finite global dimension $d$.
    \item $A$ is Gorenstein in the sense that there exists $d$ such that $\text{Ext}^d(\mathbb{K}, A) = \mathbb{K}$ and $\text{Ext}^j(\mathbb{K}, A) = 0$ for $j \neq d$.
    \item \textnormal{\( A \) has polynomial growth, i.e., \( f(n) = \dim_\mathbb{K} A_n \) is bounded above by a polynomial function of \( n \).}
\end{enumerate}
\end{definition}
By the dimension of an AS-regular algebra, we will mean the integer $d$ appearing above. If $A$ is AS-regular, then the trivial left $A$-module $A_\mathbb{K}$ has a minimal free resolution of the form
\[ 0 \rightarrow P_d \rightarrow \dots \rightarrow P_1 \rightarrow P_0 \rightarrow \mathbb{K}_A \rightarrow 0 \]
where $P_w = \bigoplus_{s=1}^{n_w} A(-i_{w,s})$ for some positive integers $n_w$ and $i_{w,s}$. The Gorenstein condition (2) implies that the above free resolution is pallindromic in the sense that the dual complex of the above sequence is a free resolution of the trivial right $A$-module (after a degree shift). As a consequence, we have $P_0 = A$, $P_d = A(-l)$, $n_w = n_{d-w}$ and $i_{w,s} + i_{d-w,n_w-s+1} = l$ for all $w, s$.




Under the natural hypothesis that $A$ is a Noetherian domain and $
\dim A = 4$, there are three possible resolution types \cite[Prop.~1.4]{Lu_2007}:
\begin{equation}\label{eq:type3}
    0 \rightarrow A(-4) \rightarrow A(-3)^4 \rightarrow A(-2)^6 \rightarrow A(-1)^4 \rightarrow A \rightarrow \mathbb{K}_A \rightarrow 0
\end{equation}
\begin{equation}
    0 \rightarrow A(-5) \rightarrow A(-4)^3 \rightarrow A(-3)^2 \oplus A(-2)^2 \rightarrow A(-1)^3 \rightarrow A \rightarrow \mathbb{K}_A\rightarrow 0
\end{equation}
\begin{equation} 
    0 \rightarrow A(-7) \rightarrow A(-6)^2 \rightarrow A(-4) \oplus A(-3) \rightarrow A(-1)^2 \rightarrow A \rightarrow \mathbb{K}_A\rightarrow 0
\end{equation}
with Hilbert series $h_A(t)$, respectively:
$$ \frac{1}{(1-t)^4}, \quad \frac{1}{(1-t)^3(1-t^2)}, \quad \frac{1}{(1-t)^2(1-t^2)(1-t^3)}.$$

Suggested by the form of the resolution as in equation (\ref{eq:type3}), we say such an algebra is of type \textit{(14641)}. In this paper we mainly deal with algebras of type \textit{(14641)}. An algebra of type \textit{(14641)} is Koszul \cite[Theorem 0.1]{MR2529094}.
So we will restrict attention to the situation where $A$ is generated in degree 1, and has defining relations in degree 2. 

\subsection{Hochschild Cohomology}

\begin{definition}
\textnormal{\textbf{Hochschild Cohomology:}
For a field  $\mathbb{K}$, Hochschild cohomology associates a sequence of $\mathbb{K}$-vector spaces \( \HH^n(A) \) to a $\mathbb{K}$-algebra \( A \). In Hochschild’s original paper \cite[Section 2]{Hochschild1945OnTC}, the \textbf{Hochschild chain complex} of \( A \) with coefficients in \( A \) are defined as
\[
C^n(A) = \text{Hom}_\mathbb{K}(A^{\otimes n}, A),
\]
where \( A^{\otimes n} \) is the tensor product of \( A \) with itself \( n \) times and \( A^{\otimes 0} = \mathbb{K} \). They are equipped with the differential
\[
d : C^n(A) \to C^{n+1}(A)
\]
defined by the following formula, for \( f \in C^n(A) \):
\begin{align*}
d f(a_1 \otimes \cdots \otimes a_{n+1}) &= a_1 f(a_2 \otimes \cdots \otimes a_{n+1}) \\
&+ \sum_{i=1}^n (-1)^i f(a_1 \otimes \cdots \otimes a_i a_{i+1} \otimes \cdots \otimes a_{n+1}) \\
&+ (-1)^{n+1} f(a_1 \otimes \cdots \otimes a_n) a_{n+1}.
\end{align*}
For the \( n = 0 \) case, we have
\[
d f(a_1) = a_1 f(1) - f(1) a_1.
\]
Thus, we can define the \( n \)-th \textbf{Hochschild cohomology} of \( A \) (with coefficients in \( A \)) as
\[
\mathrm{HH}^n(A) = \frac{\ker(d : C^n(A) \to C^{n+1}(A))}{\text{im}(d : C^{n-1}(A) \to C^n(A))}.
\]
Note that for \( n \leq -1 \), \( \mathrm{HH}^n(A) = 0 \).
}
\end{definition}

\begin{definition}
\textnormal{\textbf{Bigraded Hochschild cohomology:}
Let \( A = \bigoplus_{i \ge 0} A_i \) be a graded algebra. A standard procedure, described for example in \cite[Section 5.4]{Witherspoon2019HochschildCF}, incorporates the grading of \( A \) into its Hochschild cohomology by defining the \textbf{bigraded Hochschild complex}
\[
C_{r}^n(A) = \mathrm{Hom}_\mathbb{K}(A^{\otimes n}, A)_r \subseteq C^n(A),
\]
where \( \mathrm{Hom}_\mathbb{K}(A^{\otimes n}, A)_r \) consists of all maps \( f \in \mathrm{Hom}_\mathbb{K}(A^{\otimes n}, A) \) satisfying
\[
\deg f(a_1 \otimes \cdots \otimes a_n) = \sum_{j=1}^n \deg(a_j) + r.
\]
We can verify that the Hochschild differential \( d \) preserves this grading, i.e.,
\[
d(C_{r}^n(A)) \subseteq C_{r}^{n+1}(A).
\]
Thus, for each \( r \in \mathbb{Z} \), we define the \textbf{bigraded Hochschild cohomology} of internal degree \( r \) and cohomological degree \( n \) by
\[
\mathrm{HH}^{n}_r(A) = \frac{\ker(d : C_{r}^n(A) \to C_{r}^{n+1}(A))}{\mathrm{im}(d : C_{r}^{n-1}(A) \to C_{r}^n(A))}.
\]
The Hochschild cohomology inherits a natural internal grading:
\[
\mathrm{HH}^n(A) = \bigoplus_{r \in \mathbb{Z}} \mathrm{HH}_r^n(A).
\]}
\end{definition}



\subsection{Koszul Algebras}  
Let \( A = T(V)/\langle R \rangle \) be a connected \(\mathbb{N}\)-graded quadratic algebra over a field \( \mathbb{K} \), where \( R \subseteq V \otimes V \). Its \textbf{quadratic dual} is
\[
A^! = T(V^*) / \langle R^\perp \rangle,
\]
where \( R^\perp \subseteq V^* \otimes V^* \) is the orthogonal complement of \( R \).

Let \( h_A(t) \) and \( h_{A^!}(t) \) be the Hilbert series of \( A \) and \( A^! \), respectively.
We define  \( A \) to be \textbf{Koszul}, if the minimal graded projective resolution of the trivial \( A \)-module \( \mathbb{K} \) is linear.   Then \textbf{Koszul duality} identity holds \cite{Positselskii1995} (see also \cite[Thm.~5.9]{MR1388568}):
\[
h_A(t)\, h_{A^!}(-t) = 1.
\]

Let \(\{x_j\}\) be a basis of \(V\) with dual basis \(\{x^j\}\subset V^*\), and define the canonical tensor
\[
m = \sum_j x^j \otimes x_j \;\in V^*\otimes V.
\]
The minimal resolution of $\mathbb{K}$ is given by 
$$Q_i:=A_i^{!*}\otimes A(-i)$$ 
where \(A^{!*}_i\) is the \( \mathbb{K} \)-dual of the \( i \)-th graded component of \( A^! \).  
The differential given by multiplication by $m$.
In \cite[Prop.~3.1]{vandenBergh1994}, this element is used to construct  a resolution of \( A \) as an \( A^e \)-module (with \( A^e = A \otimes_\mathbb{K} A^{\mathrm{op}} \)) which has components
\[
P_i \cong A \otimes (A^!)_i^* \otimes A,
\]

Hence there is a chain complex $K^*$ with components
\[
K^i := \Hom_{A^e}(P_i, A) \cong \Hom_{\mathbb{K}}((A^!)_i^*, A) \cong (A^!)_i \otimes A,
\]
with differential 
\[
\delta_i(\alpha \otimes a) = \sum_j (x^j \alpha) \otimes (x_j a) - (-1)^i (\alpha x^j) \otimes (a x_j),
\qquad \alpha\in (A^!)_i,\; a\in A.
\]
which computes the Hoschschild cohomology.
This complex is bigraded by homological degree \( i \) and internal degree \( r \), inducing a decomposition
\[
\mathrm{HH}^i(A) = \bigoplus_r \mathrm{HH}^i_r(A).
\]
The \textbf{internal degree zero strand} is
\[
K^i_0 \cong (A^!)_i \otimes A_i,
\]
and the Hochschild cohomology groups in internal degree zero are computed as the cohomology of
\[
\cdots \to (A^!)_{i-1}\otimes A_{i-1} \xrightarrow{\;\delta_{i-1}\;} (A^!)_i\otimes A_i \xrightarrow{\;\delta_i\;} (A^!)_{i+1}\otimes A_{i+1}\to \cdots.
\]

\begin{example} \label{example:hochschild}
\setcounter{equation}{0}
Let 
$
A = \mathbb{K}[x_1, x_2, \ldots, x_d]
$
be the standard commutative polynomial ring over a field \( \mathbb{K} \). Then:

\begin{itemize}[leftmargin=*]
    \item \( A \) is a commutative, connected, graded, Koszul algebra generated in degree 1.  
    \item Its global dimension is \( d \).  
    \item The defining quadratic relations are commutativity:
    \[
    x_i x_j - x_j x_i = 0 \quad \text{for all } 1 \leq i < j \leq d.
    \]
    \item The quadratic dual of \( A \) is the exterior algebra on the dual basis:
    \[
    A^! = \Lambda(x^1, x^2, \ldots, x^d),
    \]
    with \( \deg x^i = 1 \) and multiplication given by anti-commutativity:
    \[
    x^i x^j = - x^j x^i, \quad (x^i)^2 = 0.
    \]
    \item As a graded vector space, we have
    \[
    \dim_{\mathbb{K}} (A^!)_n = \binom{d}{n}.
    \]
\end{itemize}

The Koszul resolution of \( A \) as an \( A^e \)-module is
\[
P_n = A \otimes (A^!)_n^* \otimes A.
\]
Applying \( \Hom_{A^e}(-, A) \) yields the Hochschild cochain complex:
\begin{equation} \label{eq:cochain complex}
K^n = \Hom_{A^e}(P_n, A) \cong \Hom_{\mathbb{K}}((A^!)_n, A).
\end{equation}

The internal degree-\( r \) component is
\[
K^n_r = \Hom_{\mathbb{K}}((A^!)_n, A_r),
\quad \text{so that} \quad
\dim_{\mathbb{K}} C^n_r = \binom{d}{n} \cdot \binom{r + d - 1}{r}.
\]

In particular, for \( r = 0 \), we have
\[
K^n_0 = \Hom_{\mathbb{K}}((A^!)_n, \mathbb{K}) \cong (A^!)_n^*,
\]
and since \( A \) is a polynomial algebra, all differentials vanish in internal degree zero. Therefore,
\[
\mathrm{HH}^n_0(A) \cong (A^!)_n,
\]
so the internal degree-zero part of Hochschild cohomology has dimension
\[
\dim_{\mathbb{K}} \mathrm{HH}^n_0(A) = \binom{d}{n}.
\]
\end{example}

For general Koszul AS-regular algebras, the dimensions of the Hochschild cohomology vary, but we have the following facts:

\begin{proposition}
Let \( A \) be a Koszul, connected, weakly Artin–Schelter regular algebra of global dimension \( d \). Then
    $$\HH^i_r(A) = 0 \mbox{ for } i < 0 \mbox{ or }i>d$$
\[
\sum_{i=0}^d (-1)^i \dim_{\mathbb{K}} \HH^i_r(A) = 
\begin{cases}
(-1)^d & \text{if } r + d = 0, \\
0 & \text{if } r + d \neq 0.
\end{cases}
\]
where \( \HH^i_r(A) \) denotes the internal degree $r$ part of \( \HH^i(A) \).
\end{proposition}
\begin{proof}
\setcounter{equation}{0} 

Since \( A \) is Koszul and weakly Artin-Schelter regular with global dimension \( d \), its Koszul dual \( A^! \) is finite-dimensional and Gorenstein.

The minimal graded projective resolution of the trivial module \(\mathbb{K}\) is the Koszul complex
\[
Q_i = A^{!*}_i \otimes A(-i).
\]
Taking the $r^{th}$ strand of this complex gives
$$\delta_0^r = \sum_{i=0}^d(-1)^i \dim A_i^{!} \dim A_{-i+r}.$$

Considering internal degree \( r \), the degree \( r \) component is \( K^i_r = A^!_i \otimes A_{r+i} \), so the Euler characteristic in degree \( r \) is
\begin{equation}\label{eq:2}
\sum_{i=0}^d (-1)^i \dim_{\mathbb K} \HH^i_r(A) = \sum_{i=0}^d (-1)^i \dim_\mathbb K (A^!_i \otimes A_{r+i})\\
= \sum_{i=0}^d (-1)^i \dim_\mathbb K (A^!_i)\dim_\mathbb{K} (A_{r+i}).
\end{equation}

By the weak AS-Gorenstein property, \(A^!\) admits a non-degenerate graded pairing
\[
A^!_i \cong A^{!*}_{d-i}.
\]

We now compute \begin{eqnarray*}
    \delta_0^r & = &  \sum_{i=0}^d(-1)^i \dim A_i^{!} \dim A_{-i+r} \\
& = & \sum_{j=0}^d(-1)^{d-j} \dim A_{d-j}^{!} \dim A_{r-d+j} \\
& = & (-1)^d\sum_{j=0}^d(-1)^j \dim A_{j}^{!} \dim A_{j+(r-d)} \\
& = & (-1)^d\sum_{j=0}^d(-1)^j \dim \HH^j_{r-d}(A) \\
\end{eqnarray*}



\end{proof}

\begin{corollary}
Let \( A \) be a Koszul, connected, weakly Artin–Schelter regular algebra of global dimension \( d \). Then
\[
\sum_{i=0}^d (-1)^i \dim_{\mathbb{K}} \HH^i_0(A) = 0,
\]
where \( \HH^i_0(A) \) denotes the internal degree zero part of \( \HH^i(A) \).
\end{corollary}








\subsection{Moduli Space}
We introduce the moduli space of AS-regular algebras of dimension four.  We will do this using the Koszul dual.  The moduli space will be the space of graded algebras with Hilbert series $(1+t)^4$ that are Frobenius.  In addition, we want these algebras to be Koszul and then their quadratic duals are AS-regular algebras, but there are infinitely many algebraic conditions for an algebra to be Koszul, so we do not impose this.

\begin{definition}
\textnormal{\cite[Tag 02ZI]{stacks-project} A \textbf{stack in groupoids} over a site $\mathcal{C}$ \cite[Tag 00VH]{stacks-project} is a category $p : \mathcal{S} \to \mathcal{C}$ over $\mathcal{C}$ such that:}
\begin{itemize}
    \item \textnormal{$p : \mathcal{S} \to \mathcal{C}$ is fibred in groupoids over $\mathcal{C}$,}
    \item \textnormal{for all $U \in \operatorname{Ob}(\mathcal{C})$, for all $x,y \in \operatorname{Ob}(\mathcal{S}_U)$, the presheaf $\operatorname{Isom}(x,y)$ is a sheaf on the site $\mathcal{C}/U$, and}
    \item \textnormal{for all coverings $\mathcal{U} = \{ U_i \to U \}$ in $\mathcal{C}$, all descent data $(x_i, \varphi_{ij})$ for $\mathcal{U}$ are effective.}
\end{itemize}
\end{definition}
Let \( \mathcal{X} \) be a stack over a base Spec $\mathbb{K}$. This means that for each open scheme $U$ in $\mathbb{K}$, the fibre \( \mathcal{X}_U \) over \( U \) is a groupoid. The objects in \( \mathcal{X}_U \) are the objects parameterized by \( U \), and the morphisms represent equivalence classes of maps between these objects.

\begin{definition}
\textnormal{\cite[Tag 026O]{stacks-project}:} Let $S$ be a base scheme contained in $\text{Sch}_{fppf}$. An \textbf{algebraic stack} over $S$ is a category 
\[
p: \mathcal{X} \to (\text{Sch}/S)_{fppf}
\]
over $(\text{Sch}/S)_{fppf}$ with the following properties:
\begin{itemize}
    \item The category $\mathcal{X}$ is a stack in groupoids over $(\text{Sch}/S)_{fppf}$.
    \item The diagonal $\Delta: \mathcal{X} \to \mathcal{X} \times \mathcal{X}$ is representable by algebraic spaces.
    \item There exists a scheme $U \in \operatorname{Ob}((\text{Sch}/S)_{\mathrm{fppf}})$ and a $1$-morphism 
    \[
    (\text{Sch}/U)_{\mathrm{fppf}} \to \mathcal{X}
    \]
    which is surjective and smooth.
\end{itemize}

\end{definition}

\begin{definition}
 A \textbf{Quotient stack} is a stack that encodes the quotient of a space by a group action while retaining information about stabilizers. Given a scheme or algebraic space \( X \) with an action of a group scheme \( G \), the quotient stack \( [X/G] \) is a category fibred in groupoids over the category of schemes.

Let $G$ be a group acting on a manifold $X$ (left action).
We define the quotient stack $[X/G]$ as
\[
[X/G](Y) := \left\{ (\pi, f) \,\middle|\,
\begin{aligned}
&\pi : P \to Y \text{ is a principal } G\text{-bundle,} \\
&f : P \to X \text{ is a } G\text{-equivariant morphism}
\end{aligned}
\right\}
\]
\textnormal{Morphisms of objects are $G$-equivariant isomorphisms.}
\end{definition}
Let \( A \) be a finite-dimensional \( \mathbb{K} \)-algebra of dimension \( d \), with a fixed basis \( \{x_1, \dots, x_d\} \). The multiplication in \( A \) is determined by structure constants \( c_{ij}^k \in \mathbb{K} \) via
\[
x_i x_j = \sum_{k=1}^d c_{ij}^k x_k.
\]
These constants define a point in the affine space \( \mathbb{A}^{d^3} \), parametrizing all such algebra structures.

The general linear group \( \mathrm{GL}(d) \) acts on \( \mathbb{A}^{d^3} \) by change of basis. For \( P \in \mathrm{GL}(d) \), the transformed structure constants \( c'_{ij}{}^k \) are given by
\[
c'_{ij}{}^k = \sum_{l,m,n=1}^d P_{il} P_{jm} (P_{nk})^{-1} c_{lm}^n.
\]

This action preserves isomorphism classes of algebras, i.e., two algebras are isomorphic if and only if their structure constants lie in the same \( \mathrm{GL}(d) \)-orbit.



We now describe closed conditions on the structure constants \( c_{ij}^k \) for certain algebraic properties.

\textbf{Case 1: Unital Algebras.}
Assume \( x_1 = 1 \) is the identity. Then for all \( i \),
\[
x_1 \cdot x_i = x_i \cdot x_1 = x_i.
\]
Using the multiplication rule \( x_i \cdot x_j = \sum_k c_{ij}^k x_k \), the unital condition imposes:
\[
c_{1i}^k = c_{i1}^k = \delta_{ik} \quad \text{for all } i,k.
\]
To preserve the identity under base change, restrict to the subgroup
\[
G(d,1) = \{ p \in GL(d) \mid p(x_1) = x_1 \} \subset GL(d).
\]

\textbf{Case 2: Associative Algebras.}
Associativity requires:
\[
(x_i \cdot x_j) \cdot x_k = x_i \cdot (x_j \cdot x_k),
\]
which translates to:
\[
\sum_{l} c_{ij}^l c_{lk}^n = \sum_{m} c_{jk}^m c_{im}^n \quad \text{for all } i,j,k,n.
\]
These are polynomial equations in the \( c_{ij}^k \) and are \( GL(d) \)-invariant.

\textbf{Case 3: Graded Algebras.}
Let $d_n = \dim A_n$. Suppose \( A = \bigoplus_{n \geq 0} A_n \) with \( \deg(x_i) = p_i \). Then:
\[
c_{ij}^k \neq 0 \Rightarrow p_k = p_i + p_j.
\]
Equivalently,
\[
c_{ij}^k = 0 \quad \text{unless} \quad \deg(x_k) = \deg(x_i) + \deg(x_j).
\]
The grading is preserved by the group \( \prod_{n > 0} GL(d_n) \), acting on \( \bigoplus_{n \geq 0} A_n \).

\textbf{Case 4: Frobenius.}
A connected graded \(\mathbb{K}\)-algebra \(A^!\) of finite dimension is \emph{Frobenius of degree \(n\)} if there is a nondegenerate pairing
\[
\langle-,-\rangle : A^!\otimes A^! \;\longmapsto\; \mathbb{K}(n) \quad \text{satisfying} \quad  
\langle a,\,b\,c\rangle = \langle a\,b,\,c\rangle.
\]

If \(A\) is AS-regular of global dimension 4, generated in degree 1 by four elements, then its Koszul dual \(A^!\) has Hilbert series \((1,4,6,4,1)\) (equivalently, a minimal resolution
\[
0\;\to\;A(-4)\;\to\;A(-3)^{4}\;\to\;A(-2)^{6}\;\to\;A(-1)^{4}\;\to\;A\;\to\;\mathbb{K}\;\to\;0,
\]
and \(h_A(t)=(1-t)^{-4}\)).  In \cite[Thm.~4.3 and Prop.~5.10]{MR1388568}, and also \cite[Cor.~D]{lu2007koszul}, it is shown that an \(A\) is AS-regular if and only if \(A^!\) is Frobenius.  So we can choose a graded bases
\[
\{\alpha_i\}_{i=1}^4,\ 
\{\beta_j\}_{j=1}^6,\ 
\{\zeta_k\}_{k=1}^4,\ 
\{\delta\},
\]
of $A_1^!,\ldots,A_4^!$ respectively.  Since $A^!$ is graded, the only nonzero products are
\[
\alpha_i\alpha_j=\sum_{k=1}^6d_{ij}^k\,\beta_k,\quad
\alpha_i\beta_j=\sum_{k=1}^4s_{ij}^k\,\zeta_k,\quad
\beta_i\alpha_j=\sum_{k=1}^4t_{ij}^k\,\zeta_k,
\]
\[
\alpha_i\zeta_j=u_{ij}\,\delta,\quad
\zeta_i\alpha_j=r_{ij}\,\delta,\quad
\beta_i\beta_j=v_{ij}\,\delta,
\]

Since $A^!$ is Frobenius, 
the matrices \(T=(t_{ij}), U =(u_{ij})\) and \(V=(v_{ij})\) are invertible.
 So we see that the Frobenius condition is open in the Zariski topology.
 Define
\[
X_4 = \left\{ c_{ij}^k \in \mathbb{A}^{16^3} \ \middle| \
\begin{array}{l}
A \text{ is unital, associative, graded, Frobenius}, \\
h_{A^!}(t) = (1 + t)^4
\end{array} \right\}.
\]

Let \( G \subseteq \prod \mathrm{GL}(A_i) \) be the group preserving the grading and unit. Then we define the moduli stack of unital, associative Frobenius algebras with Hilbert Series $(1+t)^4$ is given by
\[
\mathcal{A}_4 = [X_4 / G].
\]

The universal family is given by $G$-equivariant family over $X_4$ whose fibre over $(c_{ij}^k) \in X_4$ is the algebra with these structure constants.

A flat family of graded Frobenius algebras $A_S$ over a scheme $S$ with Hilbert Series $(1+t)^4$ is given by locally free sheaves $A_S = \bigoplus A_i$ with $\rank(A_i) = \binom{4}{i}$ for $0 \leq i \leq 4$ and a multiplication map $$A_S \otimes_S A_S \to A$$ satisfying the conditions to be unital, associative and Frobenius.

\begin{theorem}
   $\mathcal{A}_4$ is the moduli stack parameterizing the universal family of Frobenius algebras of Hilbert Series $(1+t)^4$ upto graded isomorphism. For any flat family of such algebras over a base scheme $S$, there exists a morphism $S \to \mathcal{A}_4$ such that the given family is the pullback of the universal family along this morphism.
\end{theorem}
\begin{proof} Let \( A_S \) be a flat family of graded Frobenius algebras with Hilbert series $(1+t)^4$ 
over \( S \). Flatness ensures that \( (A_S)_j \) is locally free of rank $\binom{4}{j}$ over \( S \), meaning there exists an open cover \( \{ U_i \} \) of \( S \) such that \( A_S \) is free over each \( U_i \). On each \( U_i \), the algebra \( A_S \) can be expressed as a free module, so we can choose a basis for \( A_{U_i} \), which allows us to compute the structure constants of the algebra.

These constants give a morphism \( U_i \to \mathcal{A}_4 \). Because \( \mathcal{A}_4 \) is a stack, the maps from the open cover glue together to form a global morphism \( \varphi: S \to \mathcal{A}_4 \), ensuring that \( A_S \) is the pullback of the universal family along \( \varphi \).
\end{proof}

The moduli stack contains a subset $\mathcal{KA}_4$ containing algebras that are also Koszul.  If we let $\mathcal{KAS}_4$ be the set isomorphism classes of
Koszul Artin-Schelter regular algebras of dimension four, then this set is in bijective correspondence with regular Koszul Artin-Schelter algebras of dimension four of type $14641.$

\begin{remark}
The correspondence \( R \mapsto R^\perp \), where \( R^\perp \) denotes the orthogonal complement of the relation space, defines a bijection
\[
\mathcal{KAS}_4 \underset{()^!}{\longleftrightarrow} \mathcal{KA}_4,
\]
between Koszul Artin--Schelter regular algebras of dimension four and Koszul Frobenius algebras of Hilbert Series $(1+t)^4.$

The Koszul property imposes infinite constraints on structure constants, \( \mathcal{KA}_4 \) is a subset of the space \( \mathcal{A}_4 \) of all Frobenius algebras.
We do not include the Koszul property in the definition of the moduli space since there are infinitely many conditions to ensure an algebra is Koszul.  

For a graded algebra $A$, the graded pieces $\Tor^A_i(\mathbb{K},\mathbb{K})_j$ record the internal degree of the $i$-th syzygies. If the $i$-th projective module $P^i$ is generated in degree $i$, then all syzygies in homological degree $i$ occur only in internal degree $i$. Thus the Koszul condition is
\[
\Tor^A_i(\mathbb{K},\mathbb{K})_j = 0 \quad \text{whenever } j \neq i.
\]

Since this must hold for every $i,j$, verifying Koszulity amounts to checking an infinite family of vanishing conditions.
In \cite[Theorem 5.1]{nonKoszul}, the authors construct a flat family $R_\alpha$ of non-Koszul commutative quadratic Artinian Gorenstein algebras that has a linear resolution for $\alpha$ steps for any integer $\alpha \geq 2.$
\end{remark}
One could possibly construct a better moduli space if we knew the answer to the following questions.

\begin{question}
Does there exist a graded associative Frobenius algebra with Hilbert Series $(1+t)^4$ that is not Koszul?
\end{question}

\begin{question}
    Are there finitely many algebraic conditions that ensure a graded associative Frobenius algebra with Hilbert Series $(1+t)^4$ is Koszul?
\end{question}

\section{Kodaira Spencer map}

In this section, we examine how first-order deformations of a flat family of graded algebras \( A \) over a base scheme \( S \) are detected by Hochschild cohomology. Let $p \in S$.  More specifically, we construct the \emph{Kodaira–Spencer map}, which assigns to each tangent vector \( v \in T_p S \) a class in the degree-zero component of the second Hochschild cohomology group \( \HH^2_0(A_p) \), where \( A_p \) is the fibre of \( A \) at a point \( p \in S \). This map encodes how infinitesimal variations in the base induce first-order algebra extensions of the fibre, and thus provides a key link between the geometry of the base and the deformation theory of the fibres.

\begin{definition}
\textnormal{\textbf{Hochschild Extensions:} Let $A$ be a $\mathbb{K}$-algebra, $M$ an $A$-bimodule, and \( \alpha: A \otimes A \to M \) a Hochschild 2-cocycle, i.e., a $\mathbb{K}$-linear map satisfying the Hochschild 2-cocycle condition:
\[
    a\alpha(b \otimes c) - \alpha(ab \otimes c) + \alpha(a \otimes bc) - \alpha(a \otimes b)c = 0
\]
for all \( a, b, c \in A \). The \textbf{Hochschild extension} \( T(A, M, \alpha) \) of $A$ by $M$ and $\alpha$ is the $\mathbb{K}$-algebra $A \oplus M$ with multiplication given by:
\[
    (a, m)(a', m') := (aa', am' + ma' + \alpha(a \otimes a')).
\]}
\end{definition}

\begin{theorem}\label{theorem:weibel}\cite[Theorem 9.3.1]{Weibel_1994}:
Given a $\mathbb{K}$-algebra $A$ and an $A-A$ bimodule $M$, the equivalence classes of Hochschild extensions are in 1-1 correspondence with the elements of the Hochschild cohomology module $\HH^2(A, M)$.
\end{theorem}
Let \( A \) be a flat family of graded \( \mathbb{K} \)-algebras over a base scheme \( S \). We introduce the following notation:
\begin{itemize}
    \item \( T_p S \) denotes the \textbf{Zariski tangent space} to the scheme \( S \) at a point \( p \in S \); it is defined as 
    \[
    T_p S = \operatorname{Hom}_{\mathbb{K}}(\mathfrak{m}_p / \mathfrak{m}_p^2, \mathbb{K}),
    \]
    where \( \mathfrak{m}_p \) is the maximal ideal of the local ring \( \mathcal{O}_{S,p} \).
    
    \item \( A_p \) denotes the \textbf{fibre algebra} of \( A \) at the point \( p \), given by 
    \[
    A_p = A \otimes_{S} \kappa(p),
    \]
    where \( \kappa(p) \) is the residue field at \( p \).
  
    \item \( \HH^2(A_p)_0 \) denotes the \textbf{degree zero component} of the second Hochschild cohomology group of the algebra \( A_p \).
\end{itemize}

\begin{theorem}\label{thm:KodairaSpencer} Let $S$ be a scheme over  $\mathbb{K}$-algebra and let \( p \in S \) be a closed point.  Let $A$ be a flat family of graded algebras over $S.$  Then there exists a natural linear map:
\[
\kappa_p : T_p  S \to \HH^2_0(A_p),
\]
called the \emph{Kodaira–Spencer map}.
\end{theorem}
\begin{proof}
\noindent Consider a flat family of graded algebras $A_S$ over a base scheme $S$. 
Let $\mathfrak{m}$ be the maximal ideal corresponding to the point \(p \in S \), we analyze the behavior of $A_S$ under specialization. We restrict to the local ring $B=\mathcal{O}_{S,p}$.
The filtration induced by \( \mathfrak{m} \) is the descending sequence of submodules
\[
A_B \supseteq \mathfrak{m}A_B \supseteq (\mathfrak{m}A_B)^2 \supseteq (\mathfrak{m}A_B)^3 \supseteq \dots
\]

 This filtration measures how \( A_B \) behaves under specialization at \( \mathfrak{m} \) by tracking how elements vanish modulo increasing powers of \( \mathfrak{m} \). 
\noindent The Zariski tangent space of $S$ at $\mathfrak{m}$ is given by $(\mathfrak{m}/\mathfrak{m}^2)^* = \text{Hom}_\mathbb{K}(\mathfrak{m}/\mathfrak{m}^2, \mathbb{K})$. Tensoring the short exact sequence
\[
    0 \to \mathfrak{m}/\mathfrak{m}^2 \to B/\mathfrak{m}^2 \to B/\mathfrak{m} \to 0
\]
with $A_B$ over $B$ and using flatness, we obtain
\[
0 \to \mathfrak{m}A_B / (\mathfrak{m}A_B)^2 \to A_B / (\mathfrak{m}A_B)^2 \to A_B / \mathfrak{m}A_B \to 0,
\]
which simplifies to 
\[
    0 \to A_B \otimes_B (\mathfrak{m}/\mathfrak{m}^2) \to A_B / (\mathfrak{m}A_B)^2 \to A_p \to 0.
\]
\begin{equation} \label{eq:correspondence}
0 \to A_p \otimes_\mathbb{K} (\mathfrak{m}/\mathfrak{m}^2) \to A_B/\mathfrak{m}^2A_B \to A_p \to 0
\end{equation}

From the Theorem \ref{theorem:weibel}, we see that the Hochschild extension (\ref{eq:correspondence}) gives an element $KS$ of the Hochschild Cohomology module $KS \in \HH^2(A_p, A_p \otimes_\mathbb{K} \mathfrak{m}/\mathfrak{m}^2) \cong \HH^2(A_p, A_p) \otimes \mathfrak{m}/\mathfrak{m}^2 \cong \mathrm{Hom}((\mathfrak{m}/\mathfrak{m}^2)^*, \HH^2(A_p))$.

Since $A$ is graded, we get the map:
$$T_pS = T_pB \to \HH^2_0(A_p)$$
This implies that a tangent vector in the Zariski tangent space maps to a first-order deformation of $ A_p .$
\end{proof}

\subsection{Extended Example: Deformations of Skew Polynomial Algebras}
\label{toricfamily}
We study the Hochschild cohomology and the Kodaira Spencer map in detail for the case of the family of skew polynomial rings.  The Hochschild cohomology of skew polynomial rings has been studied in \cite{grimley2016hochschild, matviichuk2024creatingquantumprojectivespaces}. We directly show that this family gives a component of the moduli space by showing that the Kodaira-Spencer map is an isomorphism for a generic member of the family.

Let $q_{ij} \in \mathbb{K}^\times$ for $1 \leq i < j \leq n$ and
for convenience, we set $q_{ii}=1$ and when $q_{ij} \neq 0,$ we set $q_{ji}=q_{ij}^{-1}$. Let $Q=(q_{ij})$ denote the multiplicatively skew-symmetric matrix.
Let $S=\mathbb{K}\langle x_1,\ldots,x_n\rangle$.
Let $I$ be the two sided ideal of $S$ generated by
$$x_jx_i - q_{ij} x_ix_j\quad \quad 1 \leq i < j \leq n.$$
We let $A = A_Q$ be $S/I$.
Note that $A$ has a $\mathbb{K}$-basis of ordered monomials and so we have
$$ A \simeq \bigoplus_{i_1\geq 0,\ldots, i_n \geq 0} \mathbb{K} x_1^{i_1}\ldots x^{i_n}_n$$
as vector spaces.  Note that the Hilbert Series of $A$ is $h_A(t) = 1/(1-t)^n$  then $A_Q$ is an AS-regular algebra.
We will compute the Hochschild cohomology
$\HH^2_0(A)$ using the degree zero part of the complex
$$ A(1)^n_0 \stackrel{d_1}{\longrightarrow} A(2)^{\binom{n}{2}}_0 \stackrel{d_2}{\longrightarrow}
  A(3)^{\binom{n}{3}}_0 $$
  We interpret $\HH^2_0(A)$ as graded first order deformations, and we first compute the kernel of $d_2$.
  
To compute the first-order deformations of \( A = S/I \), we perturb the relations. Consider \( \alpha_{ij} \in A \) for \( 1 \leq i < j \leq n \). Let \( \mathbb{K}_1 = \mathbb{K}[\varepsilon]/(\varepsilon^2) \) and define \( a_{ij} = \varepsilon \alpha_{ij} \). The perturbed relations are given by
\[
x_j x_i - q_{ij} x_i x_j - a_{ij}, \quad \text{for} \quad 1 \leq i < j \leq n.
\]
Let \( I_1 \) be the two-sided ideal generated by these relations. We then obtain the following family of algebras:
\[
A_1 = \mathbb{K}_1\langle x_1, \dots, x_n \rangle / I_1.
\]

\begin{proposition}\cite[Corollary 3.9]{MR1644207}
  The algebra $A_1$ is a  first order deformation of $A$ if and only if one of the three equivalent conditions hold:
  \begin{itemize}
    \item$A_1$ is a flat $\mathbb{K}_1$ algebra with a fixed isomorphism $A_1/(\varepsilon) \simeq A$,
    \item $\varepsilon A_1 \simeq A$ as $\mathbb{K}$-vector spaces,
    \item $(a_{ij})$ are a Hochschild cocyle giving a class in $\HH^2(A)$.
  \end{itemize}
\end{proposition}
In the deformation theory of associative algebras, flat families are constructed with the requirement that associativity is preserved. To verify this in our setting, we check that associativity holds under the perturbed relations by explicitly computing and comparing the reductions of \(  x_k(x_j x_i) \) and \((x_k x_j)x_i \) to ordered monomials for \( k > j > i \):
\begin{eqnarray*}
  x_k(x_jx_i) 
& = & x_k(q_{ij} x_ix_j + a_{ij}) = q_{ij} x_kx_ix_j+ x_k a_{ij} \\
& = & q_{ij} (q_{ik}x_ix_k + a_{ik})x_j +x_ka_{ij}\\
& = & q_{ij}q_{ik}x_ix_kx_j +q_{ij}a_{ik} x_j + x_k a_{ij}\\
& = & q_{ij}q_{ik}q_{jk} x_ix_jx_k + q_{ij}q_{ik}x_ia_{jk} + q_{ij}a_{ik}x_j + x_ka_{ij},
\end{eqnarray*}
\begin{eqnarray*}
(x_kx_j)x_i
& = & (q_{jk} x_jx_k + a_{jk})x_i = q_{jk} x_jx_kx_i+ a_{jk}x_i \\
& = & q_{jk}x_j(q_{ik}x_ix_k+a_{ik}) + a_{jk}x_i \\
& = & q_{jk}q_{ik}x_jx_ix_k+q_{jk}x_ja_{ik}+a_{jk}x_i \\
& = & q_{jk}q_{ik}q_{ij}x_ix_jx_k + q_{jk}q_{ik} a_{ij} x_k +q_{jk}x_ja_{ik} + a_{jk}x_i.
\end{eqnarray*}

Hence we have a first order deformation if and only if
$$q_{ij}q_{ik}x_ia_{jk}+q_{ij}a_{ik}x_j+x_ka_{ij} \\
= q_{jk}q_{ik} a_{ij}x_k + q_{jk}x_ja_{ik} +a_{jk}x_i$$

We will solve these equations for graded deformations, i.e.~when $\alpha_{ij} \in \varepsilon A_2$ are quadratic.

So we let
$$\alpha_{ij} = \sum_{1 \leq \ell \leq m \leq n} \alpha_{ij}^{\ell m} x_\ell x_m$$
be arbitrary quadratic elements of $A$.  We look at the above condition for these $\alpha$.
So we have \begin{align}\label{eq:deformation}
    & \sum_{1 \leq \ell \leq m \leq n}(q_{ij}q_{ik}x_i\alpha_{jk}^{\ell m} x_\ell x_m+q_{ij}\alpha_{ik}^{\ell m} x_\ell x_mx_j+x_k\alpha_{ij}^{\ell m} x_\ell x_m) = \notag \\
    & \sum_{1 \leq \ell \leq m \leq n}( q_{jk}q_{ik} \alpha_{ij}^{\ell m} x_\ell x_mx_k + q_{jk}x_j\alpha_{ik}^{\ell m} x_\ell x_m +\alpha_{jk}^{\ell m} x_\ell x_mx_i) 
\end{align}
This gives cubic expressions in $A$, which we can separate into different equations, for example:
$$(q_{ij} q_{ik} \alpha_{jk}^{ii} -\alpha_{jk}^{ii}) x_i^3 = 0$$
giving that $\alpha_{jk}^{ii} =0$ when $q_{ij} q_{ik} \neq 1,$ and $i,j,k$ are distinct.
On examining the coefficients of different monomials, we obtain the following
equations:
$$ (q_{ij}q_{ik}-q_{i\ell})\alpha_{jk}^{i\ell} = 0 \quad  \{ j,k \} \cap \{i,\ell\} = \emptyset, \quad  j<k, \quad i \leq \ell,$$
  $$ q_{ij}(q_{ik}-q_{i\ell}) \alpha_{jk}^{j\ell} + q_{ij}(q_{j\ell}-q_{jk}) \alpha_{ik}^{i\ell} = 0 \quad \quad |\{ i,j,k \}| =3.$$

We note that these conditions are mostly independent.  The first set of equations is completely independent from each other and the second set of equations.
The second set of equations decomposes into sets of three equations
for each ordered pair $k,\ell$ with $k \neq \ell.$
Hence we can combine them into a matrix equation for $i<j<k$ we have
$$\begin{pmatrix}
  q_{jk} & & \\
    & q_{ik} & \\
    & & q_{ij}
\end{pmatrix}
\begin{pmatrix}
  0 & q_{k \ell}-q_{kk} & q_{jk} - q_{j\ell} \\
  q_{kk} - q_{k \ell} & 0 & q_{i\ell} - q_{ik} \\
  q_{j\ell} - q_{jk} & q_{ik} - q_{i\ell} & 0
\end{pmatrix}
\begin{pmatrix}
  \alpha^{ik}_{i\ell} \\ \alpha^{jk}_{j\ell} \\  \alpha^{kk}_{k\ell}
  \end{pmatrix} = 0.
$$

Recall that a skew symmetric matrix always has even rank, so we always have a solution
$$\begin{pmatrix}
  \alpha^{ik}_{i\ell} \\ \alpha^{jk}_{j\ell} \\  \alpha^{kk}_{k\ell}
\end{pmatrix}
\in \mathbb{K}
\begin{pmatrix}
  q_{i\ell} - q_{ik} \\ q_{j\ell} - q_{jk} \\ q_{k\ell} - q_{kk}
\end{pmatrix}.$$
These are all solutions when the matrix has rank two, otherwise the matrix must have rank zero, so $$ q_{i\ell} = q_{ik},q_{j\ell} = q_{jk}, q_{k\ell} = q_{kk}$$
and there are no conditions on the corresponding  $\alpha^{ik}_{i\ell},\alpha^{jk}_{j\ell},\alpha^{kk}_{k\ell}.$  Hence we can conclude the following result.

\begin{definition}
Let $Q$ be a multiplicative skew symmetric matrix, then $Q$ is {\bf generic} if 
  $$q_{ij} \neq 0 \quad\text{for}\quad i<j,$$
  $$q_{ij}q_{ik} \neq q_{i\ell} \quad \text{for} \quad j<k, i \leq \ell,$$
  $$ Q \notin V(q_{i\ell} = q_{ik},q_{j\ell} = q_{jk}, q_{k\ell} = q_{kk}).$$
\end{definition}

\begin{proposition}
Let $A_1$ be determined by the $(\alpha_{ij})$ then if $Q$ is generic $A_1$ is a flat deformation if and only if the $(\alpha_{ij})$ satisfy
  $$ \alpha_{jk}^{i\ell} = 0 \quad \quad j < k, i \leq \ell,$$
$$\begin{pmatrix}
  \alpha^{ik}_{i\ell} \\ \alpha^{jk}_{j\ell} \\  \alpha^{kk}_{k\ell}
\end{pmatrix}
\in \mathbb{K}
\begin{pmatrix}
  q_{i\ell} - q_{ik} \\ q_{j\ell} - q_{jk} \\ q_{k\ell} - q_{kk}
\end{pmatrix}.$$
\end{proposition}
Since $A$ is $\mathbb{Z}^4$ graded, we can grade the Hochschild cohomology by $\mathbb{Z}^4$.  We determine the degrees of the infinitesimal deformations given above in terms of this finer grading.
\begin{definition}
\textnormal{\textbf{Grading Shift in Structure Coefficients:} Let \( A \) be a \( \mathbb{Z}^n \)-graded algebra generated by homogeneous elements \( x_1, x_2, \dots, x_n \), where each \( x_i \) is assigned a multi-degree \( \deg(x_i) = e_i \), the \( i \)-th standard basis vector of \( \mathbb{Z}^n \).}

\textnormal{For structure coefficients \( \alpha_{ij}^{kl} \) appearing in the relation
\[
x_i x_j = \alpha_{ij}^{kl} x_k x_l,
\]
the grading shift induced by \( \alpha_{ij}^{kl} \) is given by the degree difference:}
\[
\deg(\alpha_{ij}^{kl}) = \deg(x_k x_l) - \deg(x_i x_j).
\]
\textnormal{Since \( \deg(x_k x_l) = e_k + e_l \) and \( \deg(x_i x_j) = e_i + e_j \), we obtain:
\[
\deg(\alpha_{ij}^{kl}) = e_k + e_l - e_i - e_j.
\]
This vector \( \deg(\alpha_{ij}^{kl}) \in \mathbb{Z}^n \) encodes how the grading shifts when \( x_i x_j \) is rewritten in terms of \( x_k x_l \).}
\end{definition}

\begin{proposition} \label{prop:ksmapinjsurj}
Let \( A = \mathbb{K}_Q[x_1, x_2, x_3, x_4] \) be the quantum polynomial algebra associated with a generic parameter matrix \( Q = (q_{ij}) \). Then the Kodaira–Spencer map associated to the first-order deformation of \( A \) is both injective and surjective.   
\end{proposition}
\begin{proof}
In the generic case, the algebraic relations force all cross terms to vanish due to the genericity conditions on the \(q_{ij}\). The only nonzero coefficient that consistently satisfies all the relations is the diagonal coefficient \(\alpha_{ij}^{ij}\), which corresponds to the deformation of the same pair of generators.
Thus,
\[
\alpha_{ij} = \alpha_{ij}^{ij} x_i x_j
\]
This means that the only surviving deformations are those that maintain the structure of the original quadratic relations.
i.e.
\[q_{ij} \longmapsto q_{ij} + \varepsilon \alpha_{ij}^{ij}\]
Therefore, the Kodaira Spencer map at a generic point is both injective and surjective.
\end{proof}
\section{Computation of the Kodaira Spencer map}

We compiled a list of known families of AS-regular algebras of type \textit{(14641)} from \cite{Davies_2016,ExoticElliptic, grimley2016hochschild, Smith1992, VANCLIFF199463, 10.1112/jlms/jdp057, MR1429334, Pym2013PoissonSA, MR2529094, Jacobson}. The relations of each family and more precise references are in the Appendix. We compute the Hochschild cohomology and the Kodaira-Spencer map at a particular point in the family using GAP code~\cite{GAP4} by Félix Laroche, with source in \cite{dblp}.

The code computes the Hochschild cohomology as in example \ref{example:hochschild} using the Complex (\ref{eq:cochain complex})
and computes the Kodaira Spencer map as in Proposition \ref{prop:ksmapinjsurj}

Columns ``Alg" and ``Para" denote the algebra family and parameter set (\( \text{Para} = \dim T_p S \)). ``Field" indicates the base field used in the computation. The ``Evaluated at" section lists parameter values (e.g., h, p, f) from the families listed in the Appendix; $\dim \HH_0^i$ gives Hochschild cohomology dimensions. ``Inj" and ``Surj" refer to injectivity and surjectivity of the Kodaira–Spencer (K–S) map. All computations use GAP code \cite{GAP4} by Félix Laroche, with source in \cite{dblp}. 

The table includes a column, \textbf{rk}, indicating the rank of the Kodaira-Spencer map. This value satisfies the following:
\begin{itemize}
    \item The K–S Map is \textbf{injective} if and only if
    $$
    \mathrm{Rank}(\text{K–S Map}) = \dim T_p S.
    $$
    \item The K–S Map is \textbf{surjective} if and only if
    $$
    \mathrm{Rank}(\text{K–S Map}) = \dim \HH^2_0(A).
    $$
\end{itemize}

\noindent In their paper \cite{MR2529094}, the authors classify four-dimensional regular domains of the form \( (kQ[x_1, x_2])_P[y_1, y_2; \sigma] \) up to isomorphism~\cite[Proposition~4.4]{MR2529094}. Using the notion of \emph{\(\Sigma\)-\(M\)-duality}~\cite[Definition~4.3]{MR2529094}, based on twist equivalence~\cite[Definition~3.3]{MR2529094}, they identify dual pairs - (E, J), (F, I), (N, P), (T, U), and (W, Z) which are isomorphic via exchange of \( x_i \)'s and \( y_i \)'s, and self-dual algebras: B, C, M, O, R, and S.

\begin{table}[H]
    \centering
    \renewcommand{\arraystretch}{1.5}

    \begin{tabular}{|c|c|c|c|c|c|c|c|c|c|c|c|c|c|}
        \hline
        Alg & Para & \multicolumn{4}{|c|} {Evaluated at} & \multicolumn{4}{|c|} {dim HH$_{0}^{i}$} & \multicolumn{3}{|c|}{K-S Map}    \\
        \hline
        Z-Z & &Field&h&p&f&0&1&2&3& Inj &  Surj & rk \\
        \hline
        \hline
        A & 1 & $\mathbb{Q}$ & 13/27&- &- &1& 2&2&2 & Yes & No & 1\\
        \hline
        B & 1 & $\mathbb{Q}(i)$&-5/12 &$\pm i$ &- & 1 & 2 & 1 & 0 & Yes & Yes & 1\\
        \hline
        C & 1 & $\mathbb{Q}(i)$ & 81/14&$\pm i$  &-& 1 & 2 & 1 & 0 & Yes & Yes & 1\\
        \hline
        D & 2 & $\mathbb{Q}$ &5/32 &91/16 &-&  1 & 2 & 2 & 2 & Yes & Yes & 2\\
        \hline
        (E,J) & 1 &$\mathbb{Q}(i)$&43/59 &$\pm i$  &- &   1 & 2 & 1 & 0 & Yes & Yes & 1 \\
        \hline
        (F,I) & 1 &$\mathbb{Q}(i)$& 13/75&$\pm i $ &-& 1 & 2 & 1 & 0 & Yes & Yes & 1\\
        \hline
        G & 3 & $\mathbb{Q}$& 90/67&101/19& 107/7 & 1 & 2 & 2 & 2 & No & Yes & 2\\
        \hline
        H & 2 &$\mathbb{Q}$ &71/18 & -&7/12 &  1 & 3 & 3 & 1 & Yes & No & 2\\
        \hline
        K & 3 & $\mathbb{Q}$ & 121/25&14/29 & 13/2 &1 & 3 & 3 & 1 & Yes & Yes & 3\\
        \hline
        L & 3 & $\mathbb{Q}$ &7/22 &9/5 &5/11 &1 & 3 & 3 & 1 & Yes & Yes & 3\\
        \hline
        M & 2 & $\mathbb{Q}$ & 4/21&- &55/49 & 1 & 2 & 2 & 2 & Yes & Yes  & 2\\
        \hline
        N & 3 & $\mathbb{Q}$&-5/8 &11/19& 3/20 & 1 & 2 & 2 & 2 & No & Yes & 2\\
        \hline
        O & 2 & $\mathbb{Q}$& 33/42&- &-31/2 & 1 & 2 & 2 & 2 & Yes & Yes & 2\\
        \hline
        P & 2 & $\mathbb{Q}$&17/23 &- & -20/43 &1 & 2 & 2 & 2 & Yes & Yes & 2\\
        \hline
        Q & 1 & $\mathbb{Q}$ &1/7 &- &-  & 1 & 2 & 1 & 0 & Yes & Yes & 1\\
        \hline
        R & 1 & $\mathbb{Q}$&-77/190 &- & - & 1 & 2 & 1 & 0 & Yes & Yes & 1\\
        \hline
        S & 1 & $\mathbb{Q}$  &2/9 &- &- & 1 & 2 & 1 & 0 & Yes & Yes & 1\\
        \hline
        (T,U) & 1& $\mathbb{Q}$&8/11 &- & - & 1 & 2 & 1 & 0 & Yes & Yes & 1\\
        \hline
        V & 1  &$\mathbb{Q}$ & 31/49&- &- & 1 & 2 & 1 & 0 & Yes & Yes & 1\\
        \hline
        (W,Z) & 2 & $\mathbb{Q}$& 9/43&- & 43/80 & 1 & 2 & 2 & 2& Yes & Yes & 2\\
        \hline
        X & 1& $\mathbb{Q}$& 5/231&- & - & 1 & 3 & 3 & 1& Yes & No & 1 \\
        \hline
        Y & 2 & $\mathbb{Q}$ & 63/12&- &101/15 & 1 & 2 & 3 & 4 & No & No & 1\\
        \hline
        \end{tabular}
    \caption{Hochschild Cohomology Dimensions and Kodaira-Spencer Map for Koszul Artin-Schelter Regular Algebras of Dimension Four from \cite{MR2529094}}
    \label{tab:HH_dimensions_KS_map_3}
\end{table}

\begin{table}[H]
    \centering
    \renewcommand{\arraystretch}{1.1}

    \begin{tabular}{|c|c|c|c|c|c|c|c|c|c|}
        \hline
        Alg & Para & Field & \multicolumn{4}{|c|} {$\dim \HH_{0}^{i}$} & \multicolumn{3}{|c|}{K-S Map}   \\
        \hline
         &  & & 0 & 1 & 2 & 3 & Inj & Surj & rk \\
        \hline
        \hline
      Central Extensions  & 11 & $\mathbb{Q}$  & 1 & 1 & 7 & 19 & No & Yes & 7\\ \hline
      Central Extensions twist & 5 & $\mathbb{Q}$ & 1 & 1 & 4 & 9 & No & Yes& 4\\ \hline
      Sklyanin & 2 & $\mathbb{Q}$  & 1 & 1 & 2 & 9 & Yes & Yes & 2\\ \hline
      Sklyanin twist & 2 & $\mathbb{Q}$ & 1 & 1 & 8 & 21 & Yes & No & 2\\ \hline
      $S_\infty$ & 2 &$\mathbb{Q}$ & 1 & 1 & 2 & 5 & Yes & Yes & 2\\ \hline
      $S_\infty$ twist & 2 & $\mathbb{Q}$ & 1 & 1 & 8 & 17 & Yes & No & 2\\ \hline
      $S_{d,i}$  & 3 &$\mathbb{Q}$  & 1 & 1 & 3 & 9 & Yes & Yes & 3\\ \hline
      $S_{d,i}$ twist  & 3 & $\mathbb{Q}$ & 1 & 1 & 8 & 17 & Yes & No & 3\\\hline
      Vancliff   & 3 & $\mathbb{Q}$  & 1 & 3 & 4 & 3 & Yes & No & 3\\ \hline
      Vancliff twist   & 3 & $\mathbb{Q}$  & 1 & 3 & 4 & 3 & Yes & No & 3\\ \hline
      Clifford  & 24 & $\mathbb{Q}$& 1 & 1  & 9  & 19 &No & Yes & 9\\ \hline
      $\mathbb{K}[x_1, x_2, x_3, x_4]$  & 0 & $\mathbb{Q}$& 1 & 16  & 60  & 80 & Yes & No & 0 \\ \hline
      $\mathbb{K}_Q[x_1, x_2, x_3, x_4]$  & 6 & $\mathbb{Q}$ & 1 & 4 & 6 & 4 & Yes & Yes & 6\\ \hline
      $L(1,1,2)^{\sigma}$ & 5 & $\mathbb{Q}$ & 1 & 2 & 4 & 8 & No & Yes & 4\\ \hline
      $S(2,3)^{\sigma}$ & 10 & $\mathbb{Q}$ & 1 & 4 & 9 & 9 & No & Yes & 9\\ \hline
      $R(3,a)$ & 1 & $\mathbb{Q}$ & 1 & 2 & 1 & 1 & Yes & Yes &1 \\ \hline
      Shelton-Tingey 3.1& 0 & $\mathbb{Q}(i)$ & 1 & 1  & 1 & 7 & Yes & No & 0\\ \hline
      Cassidy-Vancliff 1 & 1 & $\mathbb{Q}(i)$ & 1 & 1  & 1 & 7 &Yes & Yes & 1\\ \hline
      Cassidy-Vancliff 2 & 4 & $\mathbb{Q}$ & 1 & 1  & 9 & 19 &No & No & 3\\ \hline
      Cassidy-Vancliff 3& 4 & $\mathbb{Q}$ & 1 & 1  & 9 & 7 & Yes& No & 4\\ \hline
      Kirkman R& 0 &$\mathbb{Q}$ & 1& 2& 12& 22& Yes & No & 0\\ \hline
      Kirkman S& 0 & $\mathbb{Q}$& 1& 1& 4& 9& Yes & No & 0\\ \hline
      Kirkman T& 0 &$\mathbb{Q}$& 1& 1& 4&9 & Yes & No & 0\\ \hline
      $A_5$& 4 & $\mathbb{Q}(i)$ & 1& 2& 1& 0& No& Yes & 1\\ \hline
      Caines Algebra & 4 & $\mathbb{Q}$ & 1& 1& 8& 21& No& No & 1\\ \hline
      $\mathcal{F}_{\substack{(0,-1,-1,2)\\(0,0,0,0)}}$ &5 &$\mathbb{Q}$ & 1&3&4&3&No&Yes & 4\\ \hline
$\mathcal{F}_{\substack{(-1,-1,1,1)\\(0,0,0,0)}}$ &5& $\mathbb{Q}$ & 1&3&4&3&No&Yes & 4\\ \hline
 $\mathcal{F}_{\substack{(0,-1,-1,2)\\(0,0,0,0)\\(2,-1,-1,0)}}$&5& $\mathbb{Q}$ & 1&2&4&8&No&No &3 \\ \hline
 $\mathcal{F}_{\substack{(0,-1,-1,2)\\(-1,0,-1,2)\\(0,0,0,0)}}$ 
&4& $\mathbb{Q}$ & 1&2&2&2&No&Yes & 2\\ \hline
$\mathcal{F}_{\substack{(0,-1,-1,2)\\(0,0,0,0)\\(-1,-1,2,0)}}$ 
&4& $\mathbb{Q}$ & 1&2&2&2&No&Yes & 2\\ \hline
      
    \end{tabular}
    \caption{Hochschild Cohomology Dimensions and Kodaira-Spencer Map for Known Koszul Artin-Schelter Regular Algebras of Dimension Four from \cite{MR1429334, Smith1992, Davies_2016, ExoticElliptic, 10.1112/jlms/jdp057, grimley2016hochschild, Pym2013PoissonSA, LECOUTRE_2017, Shelton2001OnKA}}
    \label{tab:HH_dimensions_KS_map_1}
\end{table}

\begin{table}[H]
    \centering
    \renewcommand{\arraystretch}{1.5}

    \begin{tabular}{|c|c|c|c|c|c|c|c|c|c|}
        \hline
        Alg & Para & Field & \multicolumn{4}{|c|} {$\dim \HH_{0}^{i}$} & \multicolumn{3}{|c|}{K-S Map}  \\
        \hline
         &  & & 0 & 1 & 2 & 3 & Inj & Surj & rk \\
        \hline
        \hline
        Lie Algebra 1  & 0 & $\mathbb{Q}$  & 1 & 9 & 36 & 53 & Yes & No & 0\\ \hline
        Lie Algebra 2  & 0 & $\mathbb{Q}$ & 1 & 7 & 23 & 32 & Yes & No & 0\\ \hline
        Lie Algebra 3  & 1 & $\mathbb{Q}$ & 1 & 6 & 21 & 31 & Yes & No & 1\\ \hline
        Lie Algebra 4  & 1 &$\mathbb{Q}$ & 1 & 6 &  21&  31&  No& No & 0\\ \hline
        $\mathfrak{sl}_2$  & 0 & $\mathbb{Q}$ & 1 & 4 & 18 & 32 & Yes & No & 0\\ \hline
        Ore Ext. Type$A_1$ & 4 & $\mathbb{Q} (\zeta_3)$ & 1 & 2 & 3 & 4 & No & Yes & 3 \\ \hline
        Ore Ext. Type$A_2$  & 4 & $\mathbb{Q} (\zeta_3)$  & 1 & 2 & 3 & 4 & No & Yes & 3 \\ \hline
        Ore Ext. Type$A_3$  & 4 & $\mathbb{Q} (\zeta_3)$  & 1 & 2 & 3 & 4 & No & Yes & 3 \\ \hline 
        Ore Ext. Type$B_1$  & 2 & $\mathbb{Q} (\zeta_2)$ & 1 & 2 & 2 & 2 & Yes & Yes & 2 \\ \hline
        Ore Ext. Type$E_1$ & 1 & $\mathbb{Q} (\zeta_9)$ & 1 & 2 & 1 & 0 & Yes & Yes & 1 \\ \hline
        Ore Ext. Type$E_2$  & 1 & $\mathbb{Q} (\zeta_9)$ & 1 & 2 & 1 & 0 & Yes & Yes & 1 \\ \hline
        Ore Ext. Type$H.$I  &  1 & $\mathbb{Q} (\zeta_4)$ & 1 & 2 & 1 & 0 & Yes & Yes & 1 \\ \hline
        Ore Ext. Type$H.$II  & 1 & $\mathbb{Q} (\zeta_4)$ & 1 & 2 & 1 & 0 & Yes & Yes & 1 \\ \hline
        Ore Ext. Type$S_1^{'}$  & 5 & $\mathbb{Q}$ & 1 & 3 & 4 & 3 & No & Yes & 4 \\ \hline
        Ore Ext. Type$S_2$& 3 & $\mathbb{Q} (\zeta_2)$ & 1 & 3 & 3 & 1 & Yes & Yes & 3\\ \hline
        Central Ext. Type$B$  & 4 & $\mathbb{Q}$ & 1 & 1 & 3 & 9 & No & Yes & 3\\ \hline
        Central Ext. Type$E$ & 0 & $\mathbb{Q}(\zeta_9)$ & 1 & 2 & 1 & 1 & Yes & No & 0\\ \hline
        Central Ext. Type$H$  &  0 & $\mathbb{Q}(\zeta_4)$  & 1 & 1 & 0 & 2 &Yes & Yes & 0 \\ \hline
        Central Ext. Type$S_1$  & 3 & $\mathbb{Q}$ & 1 & 4 & 6 & 5 & Yes & No & 3\\ \hline
        Central Ext. Type$S_1^{'}$  & 2 & $\mathbb{Q}$ & 1 & 3 & 6 & 9 & Yes & No & 2\\ \hline
        Central Ext. Type$S_2$  & 1 & $\mathbb{Q}$ & 1 & 3 & 3 & 2 & Yes & No & 1\\ \hline

    \end{tabular}
    \caption{Hochschild Cohomology Dimensions and Kodaira-Spencer Map for Known Koszul Artin-Schelter Regular Algebras of Dimension Four from \cite{Jacobson, MR1429334}}
    \label{tab:HH_dimensions_KS_map_new}
\end{table}

\section{Properties of maps}
We establish that a map from an irreducible scheme has a surjective differential at a point, then the closure of the image gives a component of target.  This maybe well known, but we provide a proof for lack of a suitable reference.

\begin{proposition}
Let \( Z \subseteq Y \) be an irreducible closed subscheme with inclusion \( i: Z \hookrightarrow Y \). If there exists \( p \in Z \) such that the differential 
\[
di_p : T_p Z \to T_{i(p)} Y
\]
is surjective, then \( Z \) is an irreducible component of \( Y \).
\label{prop:irreducible_component}
\end{proposition}

\begin{proof}
The natural morphism \( i^* \Omega_Y \to \Omega_Z \) dualizes to an exact sequence of sheaves:
\[
T_Z \to i^* T_Y \to \mathrm{coker} \to 0.
\]
Tensoring with the residue field \(\kappa(p)\) (right exact), we get
\[
T_p Z \to T_{i(p)} Y \to \mathrm{coker}_p \to 0.
\]
Surjectivity of \( di_p \) implies \(\mathrm{coker}_p = 0\). By upper semicontinuity, \( \mathrm{coker} = 0 \) on an open neighborhood \( U \subseteq Z \) containing \( p \). Hence, for the generic point \(\eta \in U\), 
\[
T_\eta Z \to T_{i(\eta)} Y
\]
is surjective.

Since \(\eta\) is smooth in \(Z\),
\[
\dim T_\eta Z = \dim Z.
\]
Surjectivity implies
\[
\dim Z = \dim T_\eta Z \geq \dim T_{i(\eta)} Y \geq \dim_\eta Y,
\]
where \(\dim_\eta Y\) denotes the local dimension of \(Y\) at \(\eta\). Because \(\eta \in Z \subseteq Y\), we also have
\[
\dim Z \leq \dim_\eta Y.
\]
Combining these,
\[
\dim_\eta Y = \dim Z,
\]
showing that \( Z \) is an irreducible component of \( Y \).
\end{proof}

\begin{proposition}\label{prop:image_irreducible}
Let \( f: X \to Y \) be a morphism of varieties, where \( X \) is irreducible. Suppose there exists a point \( p \in X \) such that the differential $d f_p : T_p X \to T_{f(p)} Y$ is surjective. Then, the closure of the image \( Z = \overline{f(X)} \) is an irreducible component of \( Y \).
\end{proposition}
\begin{proof} 
Factor \( f: X \to Y \) through its image closure \( Z \subseteq Y \):
\[
\begin{tikzcd}
X \arrow[r, "h"] \arrow[dr, "f"'] & Z \arrow[d, "i", hook] \\
& Y
\end{tikzcd}
\]
where \( h \) is dominant and \( i \) is a closed embedding.

The chain rule gives the tangent map at \( p \in X \):
\[
df_p = d(i \circ h)_p = di_{h(p)} \circ dh_p : T_p X \to T_{f(p)} Y.
\]

Since \( df_p \) is surjective by assumption, and \( dh_p \) maps to \( T_{h(p)} Z \), surjectivity of \( df_p \) forces \( di_{h(p)} \) to be surjective.

By Proposition \ref{prop:irreducible_component}, surjectivity of \( di_{h(p)} \) implies that \( Z \) is an irreducible component of \( Y \).

\end{proof}
\begin{proposition}
Let \( f \) be a map from \( X \) to \( Y \) where \( X \) and \( Y \) are schemes and there exists \( p \in X \) such that \( T_pX \) is isomorphic to \( T_{f(p)}Y \) and \( X \) is irreducible. Then $\dim \, \overline{f(X)} = \dim \, X.$
\end{proposition}
\begin{proof}
    Let \( Z = \overline{f(X)} \) be the closure of the image of \( f: X \to Y \). Factor \( f \) as
\[
\begin{tikzcd}
X \arrow[r, "h"] \arrow[dr, "f"'] & Z \arrow[d, "i", hook] \\
& Y
\end{tikzcd}
\]
where \( h \) is dominant and \( i \) is a closed embedding.

By \cite[Lemma 10.5]{H}, there exists a nonempty open subset \( U \subseteq X \) such that \( h|_U : U \to Z \) is smooth. Hence, from \cite[Tag 02G1]{stacks-project}, the sheaf of differentials \(\Omega_{X/Z}|_U\) is locally free of some rank \( n \), and \( f|_U \) has relative dimension \( n \).

For \( p \in U \), since \( df_p: T_pX \to T_{f(p)}Y \) is an isomorphism, the fibre of \( f \) at \( p \) is zero-dimensional, forcing the relative differentials to vanish:
\[
(\Omega_{X/Z})_p = 0.
\]

The exact sequence of sheaves on \( X \) is
\[
h^* \Omega_Z \to \Omega_X \to \Omega_{X/Z} \to 0.
\]

Thus, on an open \( V \subseteq X \), the sheaf \( \Omega_{X/Z} \) is locally free of rank 0 on \( U \cap V \), so the relative dimension of \( f \) is zero there.

Hence,
\[
\dim \overline{f(X)} = \dim Z = \dim X.
\]
\end{proof}
We extend the above proposition to the case where the target is a algebraic stack.
\begin{remark}
   \textnormal{The \textbf{dual numbers} form the ring $\mathbb{K}[\varepsilon] := \mathbb{K}[\varepsilon]/(\varepsilon^2)$, where $\mathbb{K}$ is a field.} 
\end{remark}
There exists a natural map
\[
i : \operatorname{Spec} \mathbb{K} \longrightarrow \operatorname{Spec} \mathbb{K}[\varepsilon],
\]
induced by the canonical quotient \( \mathbb{K}[\varepsilon] \longrightarrow \mathbb{K} \), where \( \varepsilon^2 = 0 \). This map corresponds to the inclusion of a point together with an infinitesimal extension.

\begin{definition} \label{def:tss}
\textnormal{\textbf{(Tangent space of a Stack)}: Let $\mathcal{X}$ be an algebraic stack and let $x: \operatorname{Spec} \mathbb{K} \to \mathcal{X}$ be a point. The \textbf{Zariski tangent space} of $\mathcal{X}$ at $x$, denoted $T_{\mathcal{X},x}$, is defined as the collection of commutative diagrams of the form:}
\[T_{\mathcal{X},x} = \left\{ \textnormal{2-commutative diagrams}
\begin{tikzcd}
\textnormal{Spec} \mathbb{K} \arrow[d, "i"] \arrow[dr, "x"]  \\
 \textnormal{Spec}\mathbb{K}[\varepsilon] \arrow[r, "\tau"]& \mathcal{X}
\end{tikzcd}\right\} / \mathrel{\sim}
\]
\textnormal{where $\tau: \operatorname{Spec} \mathbb{K}[\varepsilon] \to \mathcal{X}$ and $\alpha: x \xrightarrow{\sim} \tau \circ i$ is an isomorphism.} 

\textnormal{Two such pairs $(\tau, \alpha)$ and $(\tau', \alpha')$ are considered equivalent if there exists an isomorphism $\beta: \tau \xrightarrow{\sim} \tau'$ in $\mathcal{X}(\mathbb{K}[\varepsilon])$ that preserves $\alpha$, meaning $\alpha' = \beta|_{\operatorname{Spec} \mathbb{K}} \circ \alpha$.}
\end{definition}

\begin{proposition}\label{prop:surjectivity_tangent}
Let \( f: B \to \mathcal{M} \) be a morphism with \( B \) a scheme and \(\mathcal{M}\) a stack. Assume:
\begin{itemize}
    \item There is a smooth morphism \( U \to \mathcal{M} \), with \( U \) a scheme.
    \item For a closed point \( b \in B \), the induced tangent map
    \[
    df_b: T_b B \to T_{f(b)} \mathcal{M}
    \]
    is surjective.
\end{itemize}
Then the induced map on tangent spaces at \((u,b) \in U \times_{\mathcal{M}} B\),
\[
T_{(u,b)}(U \times_{\mathcal{M}} B) \to T_u U
\]
is surjective.
\end{proposition}

\begin{proof}  
From the fibre product diagram
\[
\begin{tikzcd}
U \times_{\mathcal{M}} B \arrow[r] \arrow[d] & U \arrow[d] \\
B \arrow[r, "f"] & \mathcal{M}
\end{tikzcd}
\]

we have an identification of tangent spaces:
\[
T_{(u,b)}(U \times_{\mathcal{M}} B) \cong T_u U \times_{T_{f(b)} \mathcal{M}} T_b B.
\]

Since \( U \to \mathcal{M} \) is smooth, the projection \( T_u U \to T_{f(b)} \mathcal{M} \) is surjective.

Given the surjectivity of
\[
df_b: T_b B \to T_{f(b)} \mathcal{M},
\]
the fibre product
\[
T_u U \times_{T_{f(b)} \mathcal{M}} T_b B \to T_u U
\]
is a base change of a surjective map and hence surjective.

Therefore,
\[
T_{(u,b)}(U \times_{\mathcal{M}} B) \to T_u U
\]
is surjective. 
\end{proof}
If $\mathcal{X}$ is a stack, then its topological space $|\mathcal{X}|$ \cite[Tag 04Y8]{stacks-project} is locally Noetherian \cite[Tag 0DQI]{stacks-project}. The irreducible
components of $|\mathcal{X}|$ are sometimes referred to as the irreducible components of $\mathcal{X}$.

\begin{lemma}\label{lemma:irreducible-components}\textnormal{ \cite[Tag 0DR5]{stacks-project}:}
Let \(f : U \rightarrow \mathcal{X}\) be a smooth morphism from a scheme to a locally Noetherian algebraic stack. The closure of the image of any irreducible component of \(|U|\) is an irreducible component of \(|\mathcal{X}|\). If \(U \rightarrow \mathcal{X}\) is surjective, then all irreducible components of \(|\mathcal{X}|\) are obtained in this way.
\end{lemma}
\begin{theorem}\label{thm:irreducible_component}
Let \( f: X \to \mathcal{Y} \) be a morphism with \( X \) an integral scheme and \(\mathcal{Y}\) an algebraic stack. If there exists \( p \in X \) such that the tangent map
\[
df_p: T_p X \to T_{f(p)} \mathcal{Y}
\]
is surjective, then \( \overline{f(X)} \) is an irreducible component of \(\mathcal{Y}\).
\end{theorem}

\begin{proof}  
Since \(\mathcal{Y}\) is an algebraic stack, there exists a smooth morphism \(\pi: U \to \mathcal{Y}\) from a scheme \(U\). Consider the fibre product $$ X \times_{\mathcal{Y}} U, $$
which is a scheme as a base change of \(X\).

Let \( u \in U \) with \(\pi(u) = f(p)\). By Proposition \ref{prop:surjectivity_tangent}, the induced tangent map
\[
T_{(p,u)}(X \times_{\mathcal{Y}} U) \to T_u U
\]
is surjective.

Applying Proposition \ref{prop:image_irreducible}, the closure \(\overline{(\pi \times f)(X \times_{\mathcal{Y}} U)}\) is an irreducible component of \( U \). By Lemma \ref{lemma:irreducible-components}, it follows that \(\overline{f(X)}\) is an irreducible component of \(\mathcal{Y}\).

\end{proof}

\begin{proposition}
    \textnormal{There is a natural isomorphism } $T_{[A_p]} \mathcal{A}_4 \cong \HH^2_0(A_p).$
\end{proposition}

\begin{proof} 
By Definition \ref{def:tss}, the Zariski tangent space \(T_{[A_p]} \mathcal{A}_4\) corresponds to isomorphism classes of first-order deformations of \(A_p\) over \(\operatorname{Spec} \mathbb{K}[\varepsilon]\).

By Theorem \ref{thm:KodairaSpencer}, such deformations are classified by the Kodaira–Spencer map whose image lies in \(\HH^2_0(A_p)\).

Since the universal family on \(\mathcal{A}_4\) induces an isomorphism via the Kodaira–Spencer map, we conclude that
\[
T_{[A_p]} \mathcal{A}_4 \cong \HH^2_0(A_p).
\]
\end{proof}
\begin{corollary}\label{corr510}
Let \( A \) be a flat family of unital associative Frobenius algebras with Hilbert series $(1+t)^4$ over an irreducible base scheme \( S \). Let $f: S \to \mathcal{A}_4$ be the induced map. 
If there exists a point \( p \in S \) where the Kodaira–Spencer map
\[
T_p S \to \HH^2_0(A_p)
\]
is surjective, then the closure \( \overline{f(S)} \subseteq \mathcal{A}_4 \) is an irreducible component.
\end{corollary}

\begin{theorem}\label{mainthm511}
Each family of algebras listed in Table \ref{tab:comp} maps densely onto an irreducible component of \(\mathcal{A}_4\).  Each of these components are unirational.
\end{theorem}

\begin{table}[H]
    \centering
    \renewcommand{\arraystretch}{0.9}

    \begin{tabular}{|llll|}
        \hline
        \multicolumn{4}{|c|}{Algebras} \\
        \hline \hline
        Cassidy-Vancliff 1 & B & C & D \\
        Central Extensions & (E,J) & (F,I) & G \\
        Central Extensions twist &  K & L & M \\
        \(\mathbb{K}_Q[x_1, x_2, x_3, x_4]\) & (N,P) & O & Q \\
         Sklyanin & R & S & (T,U)  \\
        \(S_{\infty}\) &  \(S_{d,i}\) & V & (W,Z)  \\
         Clifford & $R(3,a)$& $L(1,1,2)^{\sigma}$ & $S(2,3)^{\sigma}$ \\
         $A_5$ & $\mathcal{F}_{\scriptscriptstyle\substack{(0,-1,-1,2) \\ (0,0,0,0)}}$ & $\mathcal{F}_{\scriptscriptstyle\substack{(-1,-1,1,1) \\ (0,0,0,0)}}$ & $\mathcal{F}_{\scriptscriptstyle\substack{(0,-1,-1,2) \\ (-1,0,-1,2) \\ (0,0,0,0)}}$ \\
         $\mathcal{F}_{\scriptscriptstyle\substack{(0,-1,-1,2) \\ (0,0,0,0) \\ (-1,-1,2,0)}}$ &  Ore Ext. Type $A_1$  & Ore Ext. Type $A_2$ & Ore Ext. Type $A_3$ \\
        Ore Ext. Type $B_1$ & Ore Ext. Type $E_1$  &  Ore Ext. Type $E_2$    & Ore Ext. Type $H.$I \\ 
         Ore Ext. Type $H.$II &  Ore Ext. Type $S_1^{'}$  &  Ore Ext. Type $S_2$    & Central Ext. Type $B$   \\
         Central Ext. Type $H$      &     &      &     \\
       \hline
    \end{tabular}
    \caption{Families of algebras mapping densely onto irreducible components of \(\mathcal{A}_4\).}
    \label{tab:comp}

\end{table}

\begin{proof}  
By Theorem \ref{thm:irreducible_component} and the Kodaira–Spencer surjectivity data in Tables \ref{tab:HH_dimensions_KS_map_3}, \ref{tab:HH_dimensions_KS_map_1} and \ref{tab:HH_dimensions_KS_map_new}, these families map onto dense subsets of irreducible components of \(\mathcal{A}_4\). All of the parameter spaces of these families are rational, so their images are unirational.  
\end{proof}

\begin{remark} If $X$ is unirational, then there is a dominant rational map $f: \mathbb{P}^{\dim X} \dashrightarrow X$~\cite[Prop.~1.1]{unirational}.
So for each of the families in Table~\ref{tab:comp}, where the Kodaira-Spencer map is surjective but not injective, there exists a rational parametrization with $\dim \HH^2_0(A_p)$ parameters.  In particular, the family of Clifford algebras has a rational parametrization with 9 parameters.
\end{remark}

\begin{remark}\label{CentralExtensionTwistands23Twist}
The central extensions twists are given by twisting a central extension~\cite[Equation~1.1]{MR1429334} where all $l_{ij}=0$, by the diagonal automorphism $(1,1,1,-1)$.

Moreover, the algebra \( L(1,1,2)^\sigma \) is a Zhang twist~\cite{TG96} of the algebra \( L(1,1,2) \)~\cite[Section 3.2]{MR3366864} by a diagonal automorphism \( \sigma \), where \\
$
\sigma(x_0) = \alpha x_0, \quad \sigma(x_1) = \alpha x_1, \quad \sigma(x_2) = \beta x_2, \quad \sigma(x_3) = \alpha^2\beta^{-1}x_3,
$
with \( \alpha, \beta \in \mathbb{K}^\times \).

Similarly, the algebra \( S(2,3)^\sigma \) is a Zhang twist~\cite{TG96} of the algebra \( S(2,3) \)~\cite[Section 3.5]{MR3366864} by a graded automorphism \( \sigma \), where \\ 
$
\sigma(x_0) = \alpha x_0 + \beta_1 x_1 + \beta_2 x_2 + \beta_3 x_3, \quad
\sigma(x_1) = \alpha x_1, \quad
\sigma(x_2) = \alpha x_2, \quad
\sigma(x_3) = \alpha x_3,
$
with \( \alpha \in \mathbb{K}^\times, \beta_1, \beta_2, \beta_3 \in \mathbb{K} \).
\end{remark}

\begin{remark}
    We do not claim that generic algebras in these components are not isomorphic, in other words, we do not know if these components are distinct in the moduli space. We are pursuing further work that will define and compute discrete invariants to distinguish these components.
\end{remark}

\section{Appendix}
\subsection{Algebras and their Defining Relations}
The following table presents a comprehensive list of all the defining relations for the Koszul AS regular algebras listed in this paper.  

\begin{table}[H]
\centering
\renewcommand{\arraystretch}{1.1}

\label{tab:mytable1}
\makebox[\textwidth][c]{%
\rotatebox{0}{%
\begin{tabular}{|c|c|c|}
\hline
 \textbf{A} & \textbf{B} & \textbf{C} \\ \hline \hline
 None & $p^2 +1 = 0$ & $p^2 + p + 1 = 0$ \\ \hline
 $wx - xw$ & $xw - p \cdot wx$ & $zy - p \cdot yz$ \\
  $yy + yz - zy$ & $zy - p \cdot yz$ & $xw - p \cdot wx$ \\
  $yw - hwy$ & $yw +h(-xz)$ & $yw + h(wy - (p^2) \cdot xy - wz + p \cdot xz)$ \\
  $yx-h(wz+xy)$ & $yx + h(-wz)$ & $yx + h(p \cdot wy - xy - wz + p \cdot xz)$ \\
  $zw - hwz$ & $zw + h(xy)$ & $zw + h(p \cdot wy + 2 \cdot (p^2) \cdot xy - p \cdot wz + p \cdot xz)$ \\
  $zx + h(2xy + wz - xz)$ & $zx +h(-wy)$ & $zx +h( p \cdot wy - (p^2) \cdot xy - wz + xz)$ \\\hline \hline
 \textbf{D} & \textbf{E} & \textbf{F} \\ \hline \hline
 None & $p^2 + 1 = 0$ & $p^2 + 1 = 0$ \\ \hline
 $zy - p \cdot yz$ & $xw + wx$ & $xw + wx$ \\
 $yw + h(p \cdot wy)$ & $zy - p \cdot yz$ & $zy - p \cdot yz$ \\
 $yx + h((p^2) \cdot xy - wz)$ & $yw + h(-wz - xz)$ & $yw + h(wy + pxy - wz + xz)$ \\
 $zw -h( p \cdot wz)$ & $yx + h(-wz + xz)$ & $yx + h(pwy - xy - wz - xz)$ \\
 $zx +h(- wy - xz)$ & $zw + h(wy - xy)$ & $zw + h(pwy - pxy - pwz - xz)$ \\
 $xw + wx$     &   $zx + h(-wy - xy)$    &  $zx + h(pwy + pxy - wz + pxz)$     \\   \hline \hline
 \textbf{G} & \textbf{H} & \textbf{I} \\ \hline \hline
 $f \neq 0$ & None & $p^2 + 1 = 0$ \\ \hline
 $xw - wx$ & $zw-wz - ww$ & $xw - p wx$ \\
 $zy - pyz$ & $zy + yz$ & $zy + yz$ \\
 $yw + h(-pwy)$ & $yw - hwz$ & $yw + h(p wy + p xy - wz + p xz)$ \\
 $yx + h(-pwy - p^2xy - wz)$ & $yx - hfwz - hxz$ & $yx +h(- wy - xy - wz + p xz)$ \\
 $zw + h(-pwz)$ & $zw - hwy$ & $zw +h(- wy - p xy - p wz + p xz)$ \\
 $zx + h(-fwy + wz - xz)$ & $zx - hfwy - hxy$ & $zx +h( wy + p xy - wz + xz)$ \\\hline \hline
 \textbf{J} & \textbf{K} & \textbf{L} \\ \hline \hline
$p^2 + 1 = 0$ & $f \neq 0$ & $f \neq 0$ \\ \hline
$xw - p wx$ & $xw - p wx$ & $xw - p wx$ \\ 
$zy + yz$ & $zy + yz$ & $zy + yz$ \\ 
 $yw+h( - xy - xz)$ & $yw +h(- wy)$ & $yw +h(- f wz)$ \\ 
 $yx + h(wy - wz)$ & $yx +h(- xz)$ & $yx +h(- xz)$ \\ 
 $zw +h(- xy + xz)$ & $zw +h(- wz)$ & $zw +h(- f wy)$ \\ 
 $zx +h(- wy - wz)$ & $zx +h(- f xy)$ & $zx +h(- xy)$ \\ \hline
\end{tabular}
}}
\caption{Defining relations of Artin-Schelter Regular Algebras of Dimension Four from \cite{MR2529094}}
\end{table}
\begin{table}[H]
\renewcommand{\arraystretch}{1.3}

\centering
\label{tab:mytable2}
\makebox[\textwidth][c]{%
\rotatebox{0}{%
\begin{tabular}{|c|c|c|}
\hline
\textbf{M} & \textbf{N} & \textbf{O}  \\ \hline \hline
$f \neq 1$ & $f^2 \neq p^2$ & $f \neq 1$ \\ \hline
$xw + wx$ & $xw + wx$ & $xw + wx$ \\
$zy + yz$ & $zy + yz$ & $zy + yz$ \\ 
$yw + h(-xy - wz)$ & $yw + h(pxy - fxz)$ & $yw + h(-wy - fxz)$ \\ 
$yx + h(-fwy + xz)$ & $yx + h(-pwy - fwz)$ & $yx + h(xy - wz)$  \\
$zw + h(-wy + xz)$ & $zw + h(-fxy + pxz)$ & $zw + h(-fxy + wz)$  \\ 
$zx + h(xy + fwz)$ & $zx + h(-fwy - pwz)$ & $zx + h(-wy - xz)$  \\ \hline \hline
 \textbf{P} & \textbf{Q} & \textbf{R}  \\ \hline \hline
$f \neq 1$& None & None \\ \hline
$xw + wx$&$xw + wx$ & $xw + wx$ \\ 
$zy + yz$ & $zy + yz$ & $zy + yz$  \\ 
$yw + h(-wz - fxz)$ & $yw + h(-wz)$ & $yw + h(-wy - xy - wz)$  \\ 
$yx + h(-wz - xz)$ & $yx + h(-wy - xy - wz)$ & $yx + h(-wz)$  \\ 
$zw + h(-wy + fxy)$ & $zw + h(wy)$ & $zw + h(-xy)$ \\ 
$zx + h(wy - xy)$ & $zx + h(-wy + wz - xz)$ & $zx + h(xy + wz - xz)$  \\ \hline \hline
\textbf{S} & \textbf{T} & \textbf{U} \\ \hline \hline
 None & None & None \\ \hline
 $xw + wx$ & $xw + wx$ & $xw + wx$  \\ 
 $zy + yz$ & $zy + yz$ & $zy + yz$  \\ 
 $yw + h(wy - xy - wz - xz)$ & $yw + h(wy - xy - wz - xz)$ & $yw + h(wy - xy - wz - xz)$ \\ 
 $yx + h(-wy + xy - wz - xz)$ & $yx + h(-wy + xy - wz - xz)$ & $yx + h(-wy - xy - wz + xz)$  \\ 
 $zw + h(-wy - xy + wz - xz)$ & $zw + h(-wy - xy - wz + xz)$ & $zw + h(-wy - xy + wz - xz)$ \\ 
 $zx + h(-wy - xy - wz + xz)$ & $zx + h(-wy - xy + wz - xz)$  & $zx + h(-wy + xy - wz - xz)$ \\ \hline \hline
 \textbf{V} & \textbf{W} & \textbf{X} \\ \hline \hline
 None & $f \neq 1$ & None \\ \hline
 $xw - wx$ & $xw - wx$ & $xw - wx$ \\ 
 $zy + yz$ & $zy + yz$ & $zy + yz$ \\ 
$yw + h(-xy - wz)$ & $yw + h(-fxy - wz)$ & $yw + h(-wz)$ \\ 
 $yx - hxy$ & $yx + h(-wy + xz)$ & $yx + h(-wz - xz)$ \\ 
 $zw + h(wy - xy)$ & $zw + h(-wy - fxz)$ & $zw + h(-wy)$ \\
 $zx - hxz$ & $zx + h(xy - wz)$ & $zx + h(-wy - xy)$ \\ \hline
\end{tabular}
}}
\caption{Defining relations of Artin-Schelter Regular Algebras of Dimension Four from \cite{MR2529094}}
\end{table}
\begin{table}[H]
\label{tab:mytable3}
\renewcommand{\arraystretch}{1.1}
\makebox[\textwidth][c]{%
\rotatebox{0}{%
\begin{tabular}{|c|c|c|}
\hline
$\textbf{Y}$ & $\textbf{Z}$ & $\mathbb{K}_Q[x_1,x_2,x_3,x_4]$ \\ \hline \hline
\cite{MR2529094}& \cite{MR2529094} & \cite{grimley2016hochschild}\\ \hline
None & $f(f+1) \neq 0$ & $q_{ij} \in \mathbb{K^{\times}} $\\ \hline
$xw - wx$ & $yw + wy$ & $x_2x_1 - q_{12}x_1x_2$\\ 
$zy + yz$ & $zy - yz$ & $x_3x_1 - q_{13}x_1x_3$\\ 
$yw + h(-wy)$ & $yw + h(-wy - xz)$ & $x_4x_1 - q_{14}x_1x_4$\\ 
$yx + h(-fwy + xy - wz)$ & $yx + h(-xy - wz)$ & $x_3x_2 - q_{23}x_2x_3$\\ 
$zw + h(-wz)$ & $zw + h(-fxy + wz)$ & $x_4x_2 - q_{24}x_2x_4$\\ 
$zx + h(-wy - fwz + xz)$ & $zx + h(-fwy + xz)$ & $x_4x_3 - q_{34}x_3x_4$ \\ 
\hline \hline
 \textbf{Cassidy-Vancliff 1} & \textbf{Cassidy-Vancliff 2} & \textbf{Cassidy-Vancliff 3} \\ \hline \hline
 \cite[Example 1]{10.1112/jlms/jdp057} & \cite[Example 2]{10.1112/jlms/jdp057} & \cite[Example 3]{10.1112/jlms/jdp057} \\ \hline
\multicolumn{1}{|p{4.5cm}|}{\centering\shortstack{$\alpha^2 = -1 = \beta^2 $,\\ $\alpha, \beta, \gamma \in \mathbb{K^{\times}} $}} &
\multicolumn{1}{|p{4.5cm}|}{\centering\shortstack{$\alpha_2(\alpha_2-1)=0$,\\ $(\alpha_1^2 + \alpha_2^2 \beta_1)(\beta_1^2 + \beta_2^2 \alpha_1) \neq 0$}} &
\multicolumn{1}{|p{4.5cm}|}{\centering $u_{34}^2 = 1$, $u_{34} = u_{24} = u_{14}^2 = u_{13}^2$} \\ \hline

$x_4 x_1 - \alpha x_1 x_4$ & $x_3 x_1 + x_1 x_3 - \beta_2 x_2^2$ & $x_1 x_3 + u_{13} x_3 x_1$ \\ 
$x_3 x_2 - \beta x_2 x_3$ & $x_4 x_1 + x_1 x_4 - \alpha_2 x_3^2$ & $x_1 x_4 + u_{14} x_4 x_1$ \\ 
$x_3^2 - x_1^2$ & $x_2 x_3 - x_3 x_2$ & $x_3 x_4 + u_{34} x_4 x_3$ \\ 
$x_4^2 - x_2^2$ & $x_4^2 - x_2^2$ & $x_4^2 - x_2^2$ \\ 
$x_3 x_1 - x_1 x_3 + x_2^2$ & $x_4 x_2 + x_2 x_4 - x_3^2$ & $x_2 x_3 + x_3 x_2 + x_4^2$ \\ 
$x_4 x_2 - x_2 x_4 + \gamma^2 x_1^2$ & $\alpha_1 x_3^2 + \beta_1 x_2^2 - x_1^2$ & $x_2 x_4 + u_{24} x_4 x_2 + x_1^2$ \\ \hline \hline
\multicolumn{2}{|c|}{\textbf{Central Extensions}} & \textbf{Central Extensions Twist} \\
\hline \hline
\multicolumn{2}{|c|}{\cite[Equation 1.1]{MR1429334}} & 
Remark~\ref{CentralExtensionTwistands23Twist}\\ \hline
\multicolumn{2}{|c|}{$l_{ij},\ a,b, \alpha_1,\alpha_2,\alpha_3 \in \mathbb{K}$} & $q,b,c_0,c_1,c_2 \in \mathbb{K}$ \\
\hline
\multicolumn{2}{|c|}{$x_1 x_4 - x_4 x_1$} & $x_2 x_1 + q x_1 x_2 + b x_4^2 + c_0 x_3^2$ \\
\multicolumn{2}{|c|}{$x_2 x_4 - x_4 x_2$} & $x_4 x_2 + q x_2 x_4 + b x_1^2 + c_1 x_3^2$ \\
\multicolumn{2}{|c|}{$x_3 x_4 - x_4 x_3$} & $x_1 x_4 + q x_4 x_1 + b x_2^2 + c_2 x_3^2 $ \\
\multicolumn{2}{|c|}{\small $ x_1^2 + a x_2 x_3 + b x_3 x_2 + l_{11} x_1 x_4 + l_{12} x_2 x_4 + l_{13} x_3 x_4 + \alpha_1 x_4^2$} & $x_4 x_3 + x_3 x_4$ \\
\multicolumn{2}{|c|}{\small $ x_2^2 + a x_3 x_1 + b x_1 x_3 + l_{12} x_1 x_4 + l_{22} x_2 x_4 + l_{23} x_3 x_4 + \alpha_2 x_4^2$} & $x_1 x_3 + x_3 x_1 $ \\
\multicolumn{2}{|c|}{\small $ x_3^2 + a x_1 x_2 + b x_2 x_1 + l_{13} x_1 x_4 + l_{23} x_2 x_4 + l_{33} x_3 x_4 + \alpha_3 x_4^2$} & $x_2 x_3 + x_3 x_2 $ \\
\hline
\hline
\multicolumn{2}{|c|}{$\textbf{S(2,3)}^\sigma$} & \textbf{Sklyanin} \\
\hline \hline
\multicolumn{2}{|c|}{\cite[Table 8.2]{Pym2013PoissonSA}, and Remark~\ref{CentralExtensionTwistands23Twist}} & \cite[Equation 0.2.2]{Smith1992} \\ \hline

\multicolumn{2}{|c|}{\small $c_1,c_2,c_3,d_1,d_2,d_3,\alpha, \beta_1, \beta_2, \beta_3 \in \mathbb{K}$}  & 
\shortstack{$\alpha, \beta, \gamma \in \mathbb{K},\ \alpha+ \beta + \gamma +\alpha\beta\gamma = 0,$\\
$\{\alpha, \beta, \gamma\} \cap \{0, \pm1\} = \varnothing$} \\
\hline
\multicolumn{2}{|c|}{\footnotesize $\alpha(x_0 x_1 - x_1 x_0) + \beta_1 x_1^2 + \beta_2 x_2 x_1 + \beta_3 x_3 x_1 
- \alpha x_1^2 - \alpha x_1 \big( (-c_3 - 2)x_2 + c_1 x_3 \big) - d_1 \alpha x_2 x_3 $} & $f_1 := [x_0, x_1] - \alpha [x_2, x_3]_+$ \\
\multicolumn{2}{|c|}{\footnotesize $\alpha(x_0 x_2 - x_2 x_0) + \beta_1 x_1 x_2 + \beta_2 x_2^2 + \beta_3 x_3 x_2 
- \alpha x_2^2 - \alpha x_2 \big( (-c_1 - 2)x_3 + c_2 x_1 \big) - d_2 \alpha x_3 x_1$} & $f_2 := [x_0, x_1]_+ - [x_2, x_3]$ \\
\multicolumn{2}{|c|}{\footnotesize $\alpha(x_0 x_3 - x_3 x_0) + \beta_1 x_1 x_3 + \beta_2 x_2 x_3 + \beta_3 x_3^2 
- \alpha x_3^2 - \alpha x_3 \big( (-c_2 - 2)x_1 + c_3 x_2 \big) - d_3 \alpha x_1 x_2$} & $f_3 := [x_0, x_2] - \beta [x_3, x_1]_+$ \\
\multicolumn{2}{|c|}{$x_2 x_3 - x_3 x_2$} & $f_4 := [x_0, x_2]_+ - [x_3, x_1]$ \\
\multicolumn{2}{|c|}{$x_3 x_1 - x_1 x_3$} & $f_5 := [x_0, x_3] - \gamma [x_1, x_2]_+$ \\
\multicolumn{2}{|c|}{\small $x_1 x_2 - x_2 x_1$} & $f_6 := [x_0, x_3]_+ - [x_1, x_2]$ \\
\hline
\end{tabular}
}}
\caption{Defining relations of Artin-Schelter Regular Algebras of Dimension Four from \cite{MR2529094, grimley2016hochschild, 10.1112/jlms/jdp057, MR1429334, Pym2013PoissonSA, Smith1992}}
\end{table}
\begin{table}[H]
\centering
\renewcommand{\arraystretch}{0.9}
\makebox[\textwidth][c]{%
\label{tab:mytable4}
\begin{tabular}{|l|l|l|}
\hline
\textbf{Sklyanin Twist} & \textbf{Vancliff } & \textbf{Vancliff Twist} \\
\hline \hline
\cite[Lemma 4.1.3]{Davies_2016,ExoticElliptic} &  \cite[Lemma 1.13(a)]{VANCLIFF199463} & \cite[Lemma 1.13(b)]{VANCLIFF199463}\\ \hline
$\alpha, \beta, \gamma \in \mathbb{K},$ 
& $\alpha, \beta, \lambda \in \mathbb{K}^\times, \quad \lambda \neq \alpha\beta$ 
& $\alpha, \beta, \lambda \in \mathbb{K}^\times, \quad \lambda \neq \alpha\beta$ \\
$\alpha + \beta + \gamma + \alpha\beta\gamma = 0,$&&\\
$ \{ \alpha, \beta, \gamma \} \cap \{ 0, \pm 1 \} = \emptyset$&&\\
\hline
$f_1' :=[x_0, x_1]  -  \alpha [x_2, x_3]$& $x_2 x_1 - \alpha x_1 x_2$ & $x_2 x_1 - \alpha x_1 x_2$ \\
$f_2' :=[x_0, x_1]_+ - [x_2, x_3]_+$ & $x_3 x_1 - \lambda x_1 x_3$ & $x_3 x_1 - \lambda x_1 x_3$ \\
$f_3' := [x_0, x_2]  -   \beta [x_3, x_1] $& $x_4 x_1 - \alpha\lambda x_1 x_4$ & $x_4 x_1 - \alpha\lambda x_1 x_4$ \\
$f_4' := [x_0, x_2]_+ - [x_3, x_1]_+$ & $x_4 x_3 - \alpha x_3 x_4$ & $x_4 x_3 + \alpha x_3 x_4$ \\
$f_5' := [x_0, x_3]  +  \gamma [x_1, x_2]$ & $x_4 x_2 - \lambda x_2 x_4$ & $x_4 x_2 + \lambda x_2 x_4$ \\
$f_6' := [x_0, x_3]_+ + [x_1, x_2]_+$ &$ x_3 x_2 - \beta x_2 x_3 - (\alpha\beta - \lambda)x_1 x_4$ &$ x_3 x_2 + \beta x_2 x_3 - (\alpha\beta - \lambda)x_1 x_4$ \\
\hline \hline
$\textbf{L(1,1,2)}^{\sigma}$& \multicolumn{2}{l|}{$\textbf{R(3,a)}$} \\
\hline \hline 
\cite[Table 8.2]{Pym2013PoissonSA}, and Remark~\ref{CentralExtensionTwistands23Twist} & \multicolumn{2}{l|}{\cite[Example 3.20]{LECOUTRE_2017}} \\ \hline
$ p_0, p_1,\alpha, \beta, \lambda \in \mathbb{K^{\times}}; \alpha^2 = \beta \gamma$ & \multicolumn{2}{l|}{$a \in \mathbb{K} $} \\
\hline
$x_1 x_0 -  x_0 x_1$ & \multicolumn{2}{l|} { $[x_0, x_1] = -4a x_0^2$} \\ 
$\beta x_2 x_0 - p_0^{-1} \alpha x_0 x_2$ &\multicolumn{2}{l|}{\small $[x_1, x_2] = -4(a+1)x_1^2 + 8(a+1)(a+2)x_0x_1 + 4(a+2)x_0x_2$ }\\ 
$\beta x_2 x_1 - p_1 \alpha x_1 x_2 $ & \multicolumn{2}{l|}{$[x_0, x_2] = -4a x_0x_1 - 8a^2 x_0^2 + 8a x_0^2$ }\\ 
\shortstack{\small $\alpha^2\beta^{-1} x_3 x_2 - p_0^{-1} p_1 \beta x_2 x_3 \ -$\\
\small $\alpha \big( (p_1 - p_0)(x_0 x_0 + \lambda x_0 x_1 + x_1 x_1) \ +$ \\\small$  (1-p_0^2) x_0 x_0 + (p_1^2 -1) x_1 x_1 \big)$} &\multicolumn{2}{l|}{\shortstack{\small $[x_1, x_3] = -4(a+1)x_1x_2 - 8a(a+1)x_1^2 +4(a+3)x_0x_3\ +$\\$ \frac{64}{3}a(a+1)(a+2)x_0x_1  + 16(a+1)(a+2)x_0x_2 $}}\\
$\alpha^2\beta^{-1} x_3 x_0 - p_0 \alpha x_0 x_3$&\multicolumn{2}{l|}{ \small $ [x_0, x_3] = -4a x_0x_2 + 8(a - a^2)x_0x_1  + \frac{64}{6}(-a^3 + 3a^2 - 2a)x_0^2$  }\\ 
$\alpha^2\beta^{-1} x_3 x_1 - p_1^{-1} \alpha x_1 x_3 $&\multicolumn{2}{l|}{ \shortstack{ \small $ [x_2, x_3] = -4(a+2)x_2^2 + 8(a+2)(a+3)x_1x_2 +4(a+3)x_1x_3 \ -$\\\small $ \frac{64}{6}(a+2)(a+3)(a+4)x_0x_2  - 8(a+3)(a+4)x_0x_3$}} \\ 
\hline \hline
$\textbf{S}_{\infty}$ & $\textbf{S}_{\infty} \textbf{ twist}$ & $\textbf{S}_{d,i} \textbf{ twist}$ \\
\hline \hline
\cite[Equation 6.1.3]{Davies_2016,ExoticElliptic} & \cite[Equation 6.1.7]{Davies_2016,ExoticElliptic} & \cite[Equation 6.1.8]{Davies_2016,ExoticElliptic}\\ \hline
$\alpha, \beta, \gamma \in \mathbb{K},$&$\alpha, \beta, \gamma \in \mathbb{K},$&$\alpha, \beta, \gamma\in \mathbb{K}, d=(d_1, d_2) \in \mathbb{P}^1_\mathbb{K}, $\\
$\alpha + \beta + \gamma + \alpha\beta\gamma = 0,$&$\alpha + \beta + \gamma + \alpha\beta\gamma = 0,$&$\alpha + \beta + \gamma + \alpha \beta \gamma = 0,$\\
$\{\alpha, \beta, \gamma\} \cap \{0, \pm1\} = \emptyset$&$\{\alpha, \beta, \gamma\} \cap \{0, \pm1\} = \emptyset$ & $\{ \alpha, \beta, \gamma \} \cap \{ 0, \pm 1 \} = \emptyset$ \\
\hline
$[x_0, x_1]  -  \alpha[x_2, x_3]_+$& $[x_0, x_1]  -  \alpha[x_2, x_3]$ &  \\
$[x_0, x_1]_+  -  [x_2, x_3]$ & $[x_0, x_1]_+  -  [x_2, x_3]_+$ & $d_1 \Omega_1' + d_2 \Omega_2'$ \\
$[x_0, x_2]  -   \beta[x_3, x_1]_+$ & $[x_0, x_2]  -   \beta[x_3, x_1]$ & \\
$[x_0, x_2]_+  -  [x_3, x_1]$ & $[x_0, x_2]_+  -  [x_3, x_1]_+$ &$f_j^{'} : 1 \leq j \leq 6, \, j \neq i$ \\
$\Omega_1 = - x_0^2 + x_1^2 + x_2^2 + x_3^2$ & $\Omega_1 ' = - x_0^2 + x_1^2 + x_2^2 - x_3^2$ & for some $1 \le i \le 6$ \\
$\Omega_2 = x_1^2 + \frac{(1 + \alpha)}{(1 - \beta)}x_2^2 +  \frac{(1 - \alpha)}{(1 + \gamma)}x_3^2 $& $\Omega_2 ' = x_1^2 + \frac{(1 + \alpha)}{(1 - \beta)}x_2^2 -  \frac{(1 - \alpha)}{(1 + \gamma)}x_3^2$ & \\

\hline \hline
\multicolumn{2}{|c|}{\textbf{Clifford}} & $\textbf{S}_{d,i}$ \\
\hline \hline
\multicolumn{2}{|c|}{\cite[Section 1.1]{10.1112/jlms/jdp057}} & \cite[Equation 6.1.4]{Davies_2016,ExoticElliptic}\\ \hline
\multicolumn{2}{|c|}{\small $a_{ij\ell} \in \mathbb{K},\ i<j,\ 1 \le \ell \le 4$} & 
\small $\alpha, \beta, \gamma \in \mathbb{K},\ d = (d_1, d_2) \in \mathbb{P}^1_\mathbb{K}$ \\
\multicolumn{2}{|c|}{} & \small $\alpha + \beta + \gamma + \alpha\beta\gamma = 0$ \\
\multicolumn{2}{|c|}{} & \small $\{\alpha, \beta, \gamma\} \cap \{0, \pm 1\} = \varnothing$ \\
\hline
\multicolumn{2}{|c|}{$x_1x_2+x_2x_1 + a_{121}x_1^2 + a_{122}x_2^2 + a_{123}x_3^2 + a_{124}x_4^2$} &  \\
\multicolumn{2}{|c|}{$x_1x_3+x_3x_1 + a_{131}x_1^2 + a_{132}x_2^2 + a_{133}x_3^2 + a_{134}x_4^2$} & $d_1 \Omega_1 + d_2 \Omega_2$ \\
\multicolumn{2}{|c|}{$x_1x_4+x_4x_1 + a_{141}x_1^2 + a_{142}x_2^2 + a_{143}x_3^2 + a_{144}x_4^2$} &  \\
\multicolumn{2}{|c|}{$x_2x_3+x_3x_2 + a_{231}x_1^2 + a_{232}x_2^2 + a_{233}x_3^2 + a_{234}x_4^2$} & $f_j: 1 \le j \le 6, j \neq i$; \\
\multicolumn{2}{|c|}{$x_2x_4+x_4x_2 + a_{241}x_1^2 + a_{242}x_2^2 + a_{243}x_3^2 + a_{244}x_4^2$} & for some $1 \le i \le 6$ \\
\multicolumn{2}{|c|}{$x_3x_4+x_4x_3 + a_{341}x_1^2 + a_{342}x_2^2 + a_{343}x_3^2 + a_{344}x_4^2$} &  \\
\hline
\end{tabular}}
\caption{Defining relations of Artin-Schelter Regular Algebras of
Dimension Four from  \cite{DBLP:journals/jlms/CassidyV14, 10.1112/jlms/jdp057, Pym2013PoissonSA, Smith1992, CerveauLinsNeto1996, Davies_2016, ExoticElliptic, LECOUTRE_2017}}
\end{table}

\begin{table}[H]
\centering
\makebox[\textwidth][c]{%
\renewcommand{\arraystretch}{1.5}
\begin{tabular}{|c|c|c|}
\hline
\textbf{Kirkman R} & \textbf{Kirkman S} & \textbf{Kirkman T} \\ \hline \hline
\cite[page 10]{goetz2024artinschelter} & \cite[page 10]{goetz2024artinschelter} & \cite[page 10]{goetz2024artinschelter} \\
\hline 
 $x_1 x_2 + x_2 x_1$ & $x_1 x_2 - x_3 x_3$ & $x_1 x_2 - x_3 x_3$ \\
 $x_1 x_3 + x_4 x_2$ & $x_1 x_3 - x_2 x_4$ & $x_1 x_3 - x_2 x_4$ \\
 $x_1 x_4 - x_3 x_2$ & $x_1 x_4 - x_4 x_2$ & $x_1 x_4 + x_4 x_2$ \\
 $x_2 x_3 - x_4 x_1$ & $x_2 x_3 - x_3 x_1$ & $x_2 x_3 - x_3 x_1$ \\
 $x_2 x_4 + x_3 x_1$ & $x_3 x_2 - x_4 x_1$ & $x_3 x_2 - x_4 x_1$ \\
 $x_3 x_4 + x_4 x_3$ & $x_2 x_1 - x_4 x_4$ & $x_2 x_1 + x_4 x_4$ \\
 \hline \hline

$\mathcal{F}_{\substack{(0,-1,-1,2)\\(-1,0,-1,2)\\(0,0,0,0)}}$ 
&
$\mathcal{F}_{\substack{(0,-1,-1,2)\\(0,0,0,0)\\(-1,-1,2,0)}}$ 
&
\textbf{Caines Algebra}
\\
\hline \hline
$q_{12},q_{14} \in \mathbb{K}$
&
$q_{12},q_{24} \in \mathbb{K} $
&
\begin{tabular}{@{}c@{}}
\cite[Section 3.2]{MR2717252} $a, b, c, d \in \mathbb{K}$
\end{tabular}
\\
\hline
\begin{tabular}{@{}c@{}}
$x_2x_1 - q_{12}x_1x_2$\\
$x_3x_1 - q_{14}^2q_{12}^{-1}x_1x_3 - x_4^2$\\
$x_4x_1 - q_{14}x_1x_4$\\
$x_3x_2 - q_{14}^2q_{12}x_2x_3 - x_4^2$\\
$x_4x_2 - q_{14}x_2x_4$\\
$x_4x_3 - q_{14}^{-1}x_3x_4$
\end{tabular}
&
\begin{tabular}{@{}c@{}}
$x_2x_1 - q_{12}x_1x_2 - x_3^2$\\
$x_3x_1 - q_{24}^{-6}q_{12}x_1x_3$\\
$x_4x_1 - q_{24}^{-3}x_1x_4$\\
$x_3x_2 - q_{12}q_{24}^{-6}x_2x_3 - sx_4^2$\\
$x_4x_2 - q_{24}x_2x_4$\\
$x_4x_3 - q_{24}^{-1}x_3x_4$
\end{tabular}
&
\begin{tabular}{@{}c@{}}
$x_4x_3 - x_3x_4 - a x_1x_2$\\
$x_4x_2 - b x_3^2 + x_2x_4$\\
$x_4x_1 - c x_3^2 + x_1x_4$\\
$x_3x_2 - x_2x_3 + bdc^{-1}x_2x_4 - b^2dc^{-2}x_1x_4$\\
$x_3x_1 - x_1x_3 + bdc^{-1}x_1x_4 - dx_2x_4$\\
$x_2x_1 + x_1x_2$\\ 
\end{tabular}
\\

\hline
\hline 
$\mathcal{F}_{\substack{(0,-1,-1,2)\\(0,0,0,0)}}$
& $\mathcal{F}_{\substack{(-1,-1,1,1)\\(0,0,0,0)}}$ 
& $\mathcal{F}_{\substack{(0,-1,-1,2)\\(0,0,0,0)\\(2,-1,-1,0)}}$ \\ 
\hline \hline
\begin{tabular}{@{}c@{}}
$q_{13},q_{14},q_{23},q_{24} \in \mathbb{K}  $
\end{tabular}
&
\begin{tabular}{@{}c@{}}
$q_{12},q_{13},q_{14},q_{23} \in \mathbb{K} $
\end{tabular}
&
\begin{tabular}{@{}c@{}}
$q_{13},q_{23},q_{34} \in \mathbb{K} $
\end{tabular}
\\ 
\hline
\begin{tabular}{@{}c@{}}
$x_2x_1 - q_{14}^2q_{13}^{-1}x_1x_2$\\
$x_3x_1 - q_{13}x_1x_3$\\
$x_4x_1 - q_{14}x_1x_4$\\
$x_3x_2 - q_{23}x_2x_3 - x_4^2$\\
$x_4x_2 - q_{24}x_2x_4$\\
$x_4x_3 - q_{24}^{-1}x_3x_4$
\end{tabular}
&
\begin{tabular}{@{}c@{}}
$x_2x_1 - q_{12}x_1x_2 - x_3x_4$\\
$x_3x_1 - q_{13}x_1x_3$\\
$x_4x_1 - q_{14}x_1x_4$\\
$x_3x_2 - q_{23}x_2x_3$\\
$x_4x_2 - q_{13}^{-1}q_{14}^{-1}q_{23}^{-1}x_2x_4$\\
$x_4x_3 - q_{13}^{-1}q_{23}^{-1}x_3x_4$
\end{tabular}
&
\begin{tabular}{@{}c@{}}
$x_2x_1 - q_{13}^{-1}x_1x_2$\\
$x_3x_1 - q_{13}x_1x_3$\\
$x_4x_1 \pm x_1x_4$\\
$x_3x_2 - q_{23}x_2x_3 - x_4^2 - x_1^2$\\
$x_4x_2 - q_{34}^{-1}x_2x_4$\\
$x_4x_3 - q_{34}x_3x_4$
\end{tabular}
\\ 
\hline
\end{tabular}  }
\caption{Defining relations of Artin-Schelter Regular Algebras of
Dimension Four from \cite{MR2717252,  goetz2024artinschelter} }
\end{table}

\begin{table}[H]
\centering
\renewcommand{\arraystretch}{1.7}
\makebox[\textwidth][c]{%
\begin{tabular}{|c|c|c|} 
\hline
\textbf{Lie Algebra 1} & \textbf{Lie Algebra 2} & \textbf{Shelton-Tingey} \\ 
\hline 
\cite[Equation 16]{Jacobson} & \cite[Equation 17]{Jacobson} & \cite[Example 3.1]{2001_Shelton} $p^2+1 = 0; \ p \in \mathbb{K}$ \\ 
\hline
$yz - zy - xt$ & $xy - yx - xt$ & $x_3 x_1 - x_1 x_3 + x_2^2$ \\
$xy - yx$     & $xz - zx$       & $p x_4 x_1 + x_1 x_4$ \\
$xz - zx$     & $yz - zy$       & $x_4 x_2 - x_2 x_4 + x_3^2$ \\
$xt - tx$     & $xt - tx$       & $p x_3 x_2 + x_2 x_3$ \\
$yt - ty$     & $yt - ty$       & $x_1^2 - x_3^2$ \\
$zt - tz$     & $zt - tz$       & $x_2^2 - x_4^2$ \\
\hline
\hline
\textbf{Lie Algebra 3} & \textbf{Lie Algebra 4} & $\mathfrak{sl}_2$ \\
\hline
\hline
\cite[page 13]{Jacobson}  & \cite[page 13]{Jacobson}      &   \cite[Equation 22]{Jacobson}    \\ 
\hline 
\hline 
$\alpha \in \mathbb{K}$& $\beta \in \mathbb{K}$  & None \\
\hline 
$ xy-yx$& $xy-yx$ & $xz-zx- 2xt$ \\
$xz-zx- xt$& $xz-zx- xt - \beta yt $ & $yz-zy+ 2 yt$  \\
$yz-zy- \alpha yt$ & $ yz-zy-  yt$ & $xy-yx-zt$  \\
$ xt - tx$& $xt - tx$ & $xt - tx$  \\
 $yt - ty$& $yt - ty$ & $yt - ty $  \\
$zt - tz$& $zt - tz $ & $zt - tz$ \\
\hline 
\hline
 \textbf{$A_5$} & $\mathbb{K}[x_1,x_2,x_3,x_4]$ & \textbf{Central Ext. Type $S_2$}  \\ \hline \hline
 $d, a_1, a_4, a_7 \in \mathbb{K}$ & None &  $\alpha \in \mathbb{K}$  \\ \hline
 $d x_1x_4 + x_4x_1$ & $x_2x_1 - x_1x_2$ &  $x_3x_1+\alpha^{-1}x_1x_3$  \\
 $d x_2x_4 - x_4x_2$ & $x_3x_1 - x_1x_3$ & $ x_3x_2-\alpha^{-1}x_2x_3$  \\
 $a_1x_1^2 + a_4d^2x_2^2$ & $x_4x_1 - x_1x_4$  &   $ x_1^2-x_2^2$  \\
 $d x_3x_4 - i x_4x_3$ &  $x_3x_2 - x_2x_3$  &  $zx_1 -x_1z$  \\
 $a_4 x_1 x_2 + a_4 x_2 x_1 - a_7 x_3^2$ & $x_4x_2 - x_2x_4$   &  $zx_2 -x_2z$  \\
$x_2x_3 + x_3 x_2$ & $x_4x_3 - x_3x_4$   &  $zx_3 -x_3z$  \\ \hline
\end{tabular}}
\caption{Defining relations of Artin-Schelter Regular Algebras of
Dimension Four from \cite{Jacobson, MR1429334, 2001_Shelton} }
\label{tab:3by3algebras}
\end{table}

\begin{table}[H]
\centering
\renewcommand{\arraystretch}{1.1}
\makebox[\textwidth][c]{%
\begin{tabular}{|c|c|c|}
\hline
\textbf{Ore Ext. Type $A_1$} & \textbf{Ore Ext. Type $A_2$} & \textbf{Ore Ext. Type $A_3$} \\ \hline \hline
$a,b,c,d,p \in \mathbb{K}$, $p^3=1$ & $a,b,c,d,p \in \mathbb{K}$, $p^3=1$ & $a,b,c,d,p \in \mathbb{K}$, $p^3=1$ \\ \hline
\begin{tabular}{@{}c@{}}
$cx_1^{2} + ax_2x_3 + bx_3x_2$\\
$cx_2^{2} + ax_3x_1 + bx_1x_3$\\
$cx_3^{2} + ax_1x_2 + bx_2x_1$\\
$zx_1 - dx_1z$\\
$zx_2 - dp x_2z$\\
$zx_3 - dp^2x_3z$
\end{tabular}
&
\begin{tabular}{@{}c@{}}
$cx_1^{2} + ax_2x_3 + bx_3x_2$\\
$cx_2^{2} + ax_3x_1 + bx_1x_3$\\
$cx_3^{2} + ax_1x_2 + bx_2x_1$\\
$zx_1 - dx_2z$\\
$zx_2 - dp x_3z$\\
$zx_3 - dp^2x_1z$
\end{tabular}
&
\begin{tabular}{@{}c@{}}
$cx_1^{2} + ax_2x_3 + bx_3x_2$\\
$cx_2^{2} + ax_3x_1 + bx_1x_3$\\
$cx_3^{2} + ax_1x_2 + bx_2x_1$\\
$zx_1 - dx_3z$\\
$zx_2 - dp x_1z$\\
$zx_3 - dp^2x_2z$
\end{tabular}
\\ \hline \hline
\textbf{Ore Ext. Type $B_1$} & \textbf{Ore Ext. Type $E_1$} & \textbf{Ore Ext. Type $E_2$} \\ \hline \hline
$a,d,p \in \mathbb{K}$, $p^2=1$ & $d \in \mathbb{K}$ & $d \in \mathbb{K}$ \\ \hline
\begin{tabular}{@{}c@{}}
$x_1x_2 + x_2x_1 + x_2^{2} - x_3^{2}$\\
$x_1^{2} + x_2x_1 + x_1x_2 - a x_3^{2}$\\
$x_3x_1 - x_1x_3 + a x_3x_2 - a x_2x_3$\\
$zx_1 - d x_1z$\\
$zx_2 - d x_2z$\\
$zx_3 - p d x_3z$
\end{tabular}
&
\begin{tabular}{@{}c@{}}
$x_3x_1 + \zeta_9^{8}x_1x_3 + \zeta_9^{4}x_2^{2}$\\
$x_1x_2 + \zeta_9^{5}x_2x_1 + \zeta_9^{7}x_3^{2}$\\
$\zeta_9x_1^{2} + x_2x_3 + \zeta_9^{2}x_3x_2$\\
$zx_1 - d x_1z$\\
$zx_2 - d x_2z$\\
$zx_3 - d x_3z$
\end{tabular}
&
\begin{tabular}{@{}c@{}}
$x_3x_1 + \zeta_9^{8}x_1x_3 + \zeta_9^{4}x_2^{2}$\\
$x_1x_2 + \zeta_9^{5}x_2x_1 + \zeta_9^{7}x_3^{2}$\\
$\zeta_9x_1^{2} + x_2x_3 + \zeta_9^{2}x_3x_2$\\
$zx_1 - d x_1z$\\
$zx_2 - d \zeta_9^3 x_2z$\\
$zx_3 - d \zeta_9^6 x_3z$
\end{tabular}
\\ \hline \hline
\textbf{Ore Ext. Type $H$.I} & \textbf{Ore Ext. Type $H$.II} & \textbf{Ore Ext. Type $S_1'$} \\ \hline \hline
$d \in \mathbb{K}$ & $d \in \mathbb{K}$ & $\alpha,a,b,c,d \in \mathbb{K}$ \\ \hline
\begin{tabular}{@{}c@{}}
$x_1^{2} - x_2^{2}$\\
$x_1x_2 - x_2x_1 + \zeta_4 x_3^{2}$\\
$x_2x_3 - \zeta_4 x_3x_2$\\
$zx_1 - d x_1z$\\
$zx_2 + d x_2z$\\
$zx_3 - i d x_3z$
\end{tabular}
&
\begin{tabular}{@{}c@{}}
$x_1^{2} - x_2^{2}$\\
$x_1x_2 - x_2x_1 + \zeta_4 x_3^{2}$\\
$x_2x_3 - \zeta_4 x_3x_2$\\
$zx_1 - d x_1z$\\
$zx_2 - d x_2z$\\
$zx_3 + d x_3z$
\end{tabular}
&
\begin{tabular}{@{}c@{}}
$x_2x_3 + a\alpha^{-1}x_3x_2$\\
$\alpha x_3x_1 + a x_1x_3$\\
$x_3^2 + x_1x_2 + a x_2x_1$\\
$tx_1 - b^2 d x_1t$\\
$tx_2 - c^2 d x_2t$\\
$tx_3 - b c d x_3t$
\end{tabular}
\\ \hline
\textbf{Ore Ext. Type $S_2$} & \textbf{Central Ext. Type $S_1^{'}$} & \textbf{Central Ext. Type $B$} \\ \hline \hline
$\alpha,a,d,p \in \mathbb{K}$, $p^2=1$ & $a, \alpha \in \mathbb{K}$ & $a,l_{11}, l_{12}, l_{22} \in \mathbb{K}$ \\ \hline
\begin{tabular}{@{}c@{}}
$x_3x_1+\alpha^{-1}x_1x_3$\\
$x_3x_2-\alpha^{-1}x_2x_3$\\
$x_1^2-x_2^2$\\
$tx_1 - a d x_1t$\\
$tx_2 - p a d x_2t$\\
$tx_3 - d x_3t$
\end{tabular}
&
\begin{tabular}{@{}c@{}}
$x_2x_3+a\alpha^{-1}x_3x_2 $ \\
$\alpha x_3x_1+ax_1x_3 $ \\
$x_3^2+x_1x_2+ax_2x_1 $ \\
$zx_1 - x_1z$\\
$zx_2 - x_2z$\\
$zx_3 - x_3z$
\end{tabular}
&
\begin{tabular}{@{}c@{}}
$x_1x_2 + x_2x_1 + x_2^{2} - x_3^{2} + (l_{11}x_1 + l_{12}x_2)z$\\
$x_1^{2} + x_2x_1 + x_1x_2 - a x_3^{2} + (l_{12}x_1 + l_{22}x_2)z$\\
$x_3x_1 - x_1x_3 + a x_3x_2 - a x_2x_3$\\
$zx_1 - x_1z$\\
$zx_2 - x_2z$\\
$zx_3 - x_3z$
\end{tabular}
\\ \hline
\textbf{Central Ext. Type $E$} & \textbf{Central Ext. Type $H$} & \textbf{Central Ext. Type $S_1$} \\ \hline \hline
 &  & $a,\alpha,\beta \in \mathbb{K}$ \\ \hline
\begin{tabular}{@{}c@{}}
$x_3x_1 + \zeta_9^{8}x_1x_3 + \zeta_9^{4}x_2^{2}$\\
$x_1x_2 + \zeta_9^{5}x_2x_1 + \zeta_9^{7}x_3^{2}$\\
$\zeta_9x_1^{2} + x_2x_3 + \zeta_9^{2}x_3x_2$\\
$zx_1 - x_1z$\\
$zx_2 - x_2z$\\
$zx_3 - x_3z$
\end{tabular}
&
\begin{tabular}{@{}c@{}}
$x_1^{2} - x_2^{2} + x_1z$\\
$x_1x_2 - x_2x_1 + \zeta_4 x_3^{2}$\\
$x_2x_3 - \zeta_4 x_3x_2$\\
$zx_1 - x_1z$\\
$zx_2 - x_2z$\\
$zx_3 - x_3z$
\end{tabular}
&
\begin{tabular}{@{}c@{}}
$x_2x_3 + a \beta x_3x_2$\\
$\alpha x_3x_1 + a x_1x_3$\\
$x_1x_2 + a x_2x_1$\\
$zx_1 - x_1z$\\
$zx_2 - x_2z$\\
$zx_3 - x_3z$
\end{tabular}
\\ \hline
\end{tabular}}
\caption{Defining relations of Artin-Schelter Regular Algebras
of Dimension Four from \cite{MR1429334} }
\end{table}

\pagestyle{plain}  
\printbibliography[title={References}]


\end{document}